\documentclass[11pt]{amsart}

\usepackage{amssymb}
\usepackage[colorlinks=true, linkcolor=blue, anchorcolor=blue, citecolor=blue, filecolor=blue, menucolor= blue, urlcolor=red]{hyperref}
\usepackage{xcolor}
\usepackage{pgf,tikz}
\usetikzlibrary{arrows,shapes.geometric}

\numberwithin{equation}{section}
\numberwithin{figure}{section}
\newtheorem{thm}{Theorem}[section]
\newtheorem{lem}[thm]{Lemma}
\newtheorem{cor}[thm]{Corollary}

\theoremstyle{definition}
\newtheorem{defn}[thm]{Definition}
\newtheorem{defn-rem}[thm]{Definition and Remark}
\newtheorem{exmp}[thm]{Example}

\newtheorem{rem}[thm]{Remark}

\newtheorem{ques}[thm]{Question}

\def\X{{\mathbb X}}
\def\L{{\mathbb L}}

\def\H{{\text {\bf H}}}

\def\ra{\longrightarrow}

\def\Z{{\mathbb Z}}
\def\Y{{\mathbb Y}}
\def\P{{\mathbb P}}

\def\HH{{\mathbb H}}

\def\ra{\rightarrow}

\def\L{{\mathbb L}}

\def\ds{\displaystyle}

\def\k{\Bbbk}

\colorlet{myGray}{gray!25}

\tikzset{my circle/.style={circle,draw=black,fill=myGray,inner
    sep=0pt,minimum size=6pt}} \tikzset{my square/.style={regular
    polygon,regular polygon sides=4,draw=black,fill=myGray,inner
    sep=0pt,minimum size=9pt}} \tikzset{my star/.style={star,star
    point ratio=2.5,draw=black,fill=myGray,inner sep=0pt,minimum
    size=9pt}} \tikzset{my triangle/.style={regular polygon,regular
    polygon sides=3,draw=black,fill=myGray,inner sep=0pt,minimum
    size=9pt}} \tikzset{my
  kite/.style={diamond,aspect=0.25,draw=black,fill=myGray,inner
    sep=1pt,minimum size=6pt}}
\tikzset{every pin/.style={pin distance=3pt,inner sep=1pt,font=\tiny}}
\tikzset{every pin edge/.style={semithick}}

\renewcommand{\geq}{\geqslant}
\renewcommand{\leq}{\leqslant}

\begin{document}


\title{Distinguishing $\k$-configurations} 
\thanks{Last updated: \today}

\author[F. Galetto]{Federico Galetto}
\address{Department of Mathematics and Statistics\\
  McMaster University, Hamilton, ON, L8S 4L8}
\email{galettof@math.mcmaster.ca}

\author[Y.S. SHIN]{Yong-Su Shin} \address{Department of Mathematics,
  Sungshin Women's University, Seoul, Korea, 136-742}
\email{ysshin@sungshin.ac.kr }

\author[A. Van Tuyl]{Adam Van Tuyl}
\address{Department of Mathematics and Statistics\\
  McMaster University, Hamilton, ON, L8S 4L8}
\email{vantuyl@math.mcmaster.ca}

\dedicatory{Dedicated to the memory of A.V. Geramita}


\begin{abstract}   A $\k$-configuration is a set of points 
$\X$ in $\mathbb{P}^2$ that satisfies a number of geometric conditions.
Associated to a $\k$-configuration is 
a sequence $(d_1,\ldots,d_s)$ of positive integers,
called its type, which encodes many of its homological invariants.
We distinguish $\k$-configurations
by counting the number of lines that contain $d_s$ points of
$\X$.  In particular, we show that for all integers $m \gg 0$,
the number of such lines 
is precisely the value of $\Delta \H_{m\X}(m d_s -1)$. Here,
$\Delta \H_{m\X}(-)$ is the 
first difference of the Hilbert function of
the fat points of multiplicity $m$ supported on $\X$.  
\end{abstract}

\maketitle


\section{Introduction}\label{sec:intro}

In the late 1980's, Roberts and Roitman \cite{RR} introduced
special configurations of points in $\mathbb{P}^2$ which they named
$\k$-configurations.  We recall this definition:

\begin{defn}\label{kdefn}
  A $\k$-configuration of points in $\mathbb{P}^2$ is a finite set
  $\mathbb{X}$ of points in $\mathbb{P}^2$ which satisfies the following
  conditions: there exist integers $1\leqslant d_1 < \cdots < d_s$,
  subsets $\mathbb{X}_1, \ldots, \mathbb{X}_s$ of $\mathbb{X}$, and
  distinct lines
  $\mathbb{L}_1, \ldots, \mathbb{L}_s \subseteq \mathbb{P}^2$ such
  that:
  \begin{enumerate}
  \item $\mathbb{X} = \bigcup_{i=1}^s \mathbb{X}_i$;
  \item $|\mathbb{X}_i| = d_i$ and
    $\mathbb{X}_i \subseteq \mathbb{L}_i$ for each $i=1,\ldots,s$,
    and;
  \item $\mathbb{L}_i$ ($1< i \leqslant s$) does not contain any
    points of $\mathbb{X}_j$ for all $1\leq j<i$.
  \end{enumerate}
In this case, the $\k$-configuration is said to be
of type $(d_1,\ldots,d_s)$.
\end{defn}

\noindent
This definition was first extended to $\mathbb{P}^3$ by Harima \cite{H},
and later to all $\mathbb{P}^n$ by Geramita, Harima, and
Shin (see \cite{GHS:1,GHS:5}).   As shown by Roberts and
Roitman \cite[Theorem 1.2]{RR}, all $\k$-configurations
of type $(d_1,\ldots,d_s)$ have the same Hilbert function (which can
be computed from the type).  This result was later generalized by
Geramita, Harima, and Shin \cite[Corollary 3.7]{GHS:2} to show that all the
graded Betti numbers of the associated graded ideal $I_\X$ only depend
upon the type.

Interestingly, $\k$-configurations of the same type can have very
different geometric properties.  Figure \ref{fig:kconfig-123} shows
various examples of $\k$-configurations of type $(1,2,3)$. Note that 
the different shapes correspond to different sets of $\X_i$, i.e.,
the star is the point of $\X_1$, the squares are the two points of $\X_2$,
and the circles are the three points $\X_3$.
\begin{figure}[ht]
  \centering
  \begin{tikzpicture}[scale=0.35]
  \node at (-2,6) {$\L_1$};
  \node at (-2,0) {$\L_2$};
  \node at (-1,-2) {$\L_3$};
  \clip (-1,-1) rectangle (10,8);
  \draw (-1,0)--(10,0);
  \draw (-1,-1)--(10,10);
  \draw (-2,6)--(13,-3); 
  \draw[dashed] (2,-6)--(8,12); 
  \node[my circle] at (0,0) {};
  \node[my circle] at (3,3) {};
  \node[my circle] at (6,6) {};
  \node[my square] at (4,0) {};
  \node[my square] at (8,0) {};
  \node[my star] at (4.66,2) {};
\end{tikzpicture}\qquad\qquad
  \begin{tikzpicture}[scale=0.35]
  \node at (-2,6) {$\L_1$};
  \node at (-2,0) {$\L_2$};
  \node at (-1,-2) {$\L_3$};
  \clip (-1,-1) rectangle (10,8);
  \draw (-1,0)--(10,0);
  \draw (-1,-1)--(10,10);
  \draw (-2,6)--(13,-3); 
  \draw[dashed] (2,-6)--(8,12); 
  \node[my circle] at (0,0) {};
  \node[my circle] at (3,3) {};
  \node[my circle] at (6,6) {};
  \node[my square] at (4,0) {};
  \node[my square] at (8,0) {};
  \node[my star] at (5.5,1.5) {};
\end{tikzpicture}
\bigskip{}

  \begin{tikzpicture}[scale=0.35]
  \node at (-2,6) {$\L_1$};
  \node at (-2,0) {$\L_2$};
  \node at (-1,-2) {$\L_3$};
  \clip (-1,-1) rectangle (10,8);
  \draw (-1,0)--(10,0);
  \draw (-1,-1)--(10,10);
  \draw (-2,6)--(13,-3); 
  \draw[dashed] (2,-6)--(8,12); 
  \node[my circle] at (0,0) {};
  \node[my circle] at (3,3) {};
  \node[my circle] at (6,6) {};
  \node[my square] at (4,0) {};
  \node[my square] at (6.5,0) {};
  \node[my star] at (5.5,1.5) {};
\end{tikzpicture}\qquad\qquad
  \begin{tikzpicture}[scale=0.35]
  \node at (-2,6) {$\L_1$};
  \node at (-2,0) {$\L_2$};
  \node at (-1,-2) {$\L_3$};
  \clip (-1,-1) rectangle (10,8);
  \draw (-1,0)--(10,0);
  \draw (-1,-1)--(10,10);
  \draw (-2,6)--(13,-3); 
  \draw[dashed] (2,-6)--(8,12); 
  \node[my circle] at (1,1) {};
  \node[my circle] at (3,3) {};
  \node[my circle] at (6,6) {};
  \node[my square] at (4,0) {};
  \node[my square] at (6.5,0) {};
  \node[my star] at (5.5,1.5) {};
\end{tikzpicture}

\caption{Four different $\k$-configurations of type (1,2,3).}
\label{fig:kconfig-123}
\end{figure}
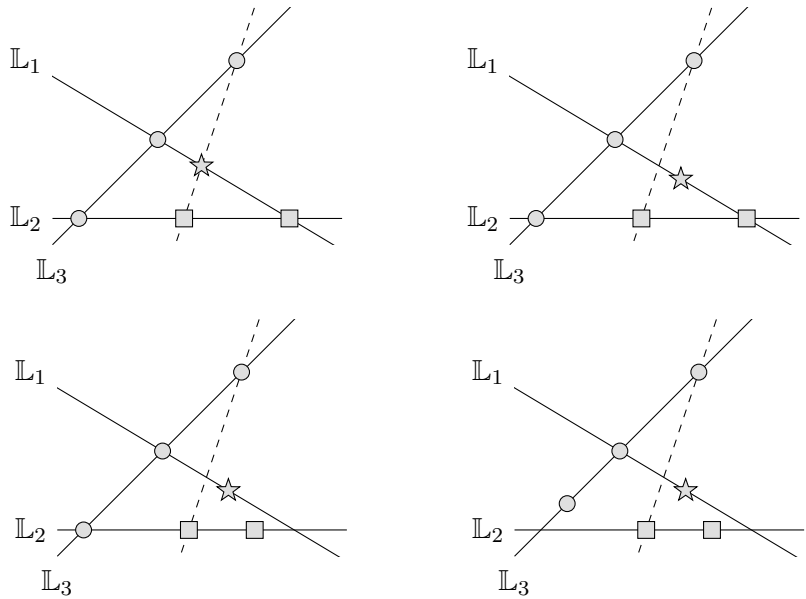
From a geometric point-of-view, these configurations
are all qualitatively different in that the number of lines containing 
$3$ points in each configuration is different (e.g., 
there are four lines that contain $3$ points of $\X$ in the first
configuration, but only one such line in the last configuration). However,
from an algebraic point-of-view, because these sets of points
are all $\k$-configurations of type $(1,2,3)$, the graded resolutions
(and consequently, the Hilbert functions) of these sets of points
are all the same.  So the algebra does not ``see'' these lines, and
so we cannot distinguish these $\k$-configurations.

Our goal in this paper is to determine how one
can distinguish these $\k$-configurations
from an algebraic point-of-view.  In particular, we wish to distinguish
$\k$-configurations by the number of lines that contain $d_s$
points of $\X$.  
It can be shown (see Remark \ref{hilbertfcnreducedbnd})
that the first difference function of the Hilbert function of $\X$
produces only an upper bound on the number of lines.  We 
show that one can obtain an exact value if one instead 
considers the Hilbert function
of the set of fat
points supported on the $\k$-configuration.  Precisely, we prove:

\begin{thm}\label{maintheorem}
  Let $\X \subseteq \mathbb{P}^2$ be a $\k$-configuration of type
  $d=(d_1,\dots,d_s) \neq (1)$.   Then there exists an integer $m_0$ such
that for all $m \geq m_0$,
\[\Delta \H_{m\X}(md_s-1)  = \mbox{number
      of lines containing exactly $d_s$ points of $\X$,}\]
where $\Delta \H_{m\X}(-)$ is the first difference function of the Hilbert
function of fat points of multiplicity $m$ supported on $\X$.  Furthermore,
if $d_s > s$, then $m_0 = 2$, and if $d_s = s$, then $m_0 = s+1$.
\end{thm}

\noindent
In other words, the number of lines that contain $d_s$ points of $\X$
is encoded in the Hilbert function of fat points supported on $\X$.
This provides us with an algebraic method to differentiate
$\k$-configurations.  Note we exclude the $\k$-configuration of type
$d=(1)$ since $\X$ is a single point, and there are an infinite number
of lines through this point.  Thematically, this paper is similar to
works of Bigatti, Geramita, and Migliore \cite{BGM}, and Chiantini and Migliore \cite{CM} which derived
geometric consequences about points from the Hilbert function.

We now give an outline of the paper.  In Section 2, we define all the
relevant terminology involving $\k$-configurations, and  some properties
of $\k$-configurations.  We also recall a
procedure to bound values of the Hilbert function of a set of fat
points due to Cooper, Harbourne, and Teitler \cite{CHT}, which will be our
main tool. In Section 3, we focus on the case $d_s > s$
and prove Theorem \ref{maintheorem} in this case. In Section 4,
we focus on the case that $d_s = s$.  A more subtle argument 
is needed to prove Theorem \ref{maintheorem} since an extra line
may come into play.   We will also require a result of
Catalisano, Trung, and Valla \cite{CTV} to complete this case.
In the final section, we give a reformulation of Theorem \ref{maintheorem}, and make a connection to a question of Geramita, Migliore, and Sabourin \cite{GMS} on the number of Hilbert functions of fat points whose support has a fixed Hilbert function.


\section{Background results}\label{sec:background} 

This section collects the necessary background results.
We first review Hilbert functions and ideals of (fat) points 
in $\mathbb{P}^2$. We then introduce a number of lemmas 
describing $\k$-configurations.
Throughout the remainder of this paper, $R = \k[x_0,x_1,x_2]$ is
a polynomial ring over an algebraically closed field $\k$. 

\subsection{Points in \texorpdfstring{$\mathbb{P}^2$}{P2} and Hilbert functions}
We recall some general facts about (fat) points
in $\mathbb{P}^2$ and their Hilbert functions.   These results will
be used later in our study of $\k$-configurations.

Let $\X = \{P_1,\ldots,P_s\}$ be a set of distinct points in
$\mathbb{P}^2$.  If $I_{P_i}$ is the ideal associated to $P_i$ in $R =
\k[x_0,x_1,x_2]$, then the homogeneous ideal associated to
$\X$ is the ideal $I_\X = I_{P_1} \cap \cdots \cap I_{P_s}$. 
Given $s$ positive integers $m_1,\ldots,m_s$ (not necessarily distinct),
the scheme defined by the ideal 
{
$I_\Z = I_{P_1}^{m_1} \cap \cdots \cap I_{P_s}^{m_s}$} is called a set of 
{\it fat points}.   We say that $m_i$ is the {\it  multiplicity}
of the point $P_i$.  If $m_1 = \cdots = m_s = m$, then we say
$\Z$ is a {\it homogeneous set of fat points} of multiplicity $m$.  In this
case, we normally write $m\X$ for $\Z$, and $I_{m\X}$ for $I_\Z$.

Note that it can be shown that $I_{m\X} = I_{\X}^{(m)}$, the $m$-th
symbolic power of the ideal $I_\X$.  If $\X = \{P\}$, then we
sometimes write $I_{mP}$ for $I_{m\X}$.  As well, since $I_P$ is a
complete intersection, it follows that $I_{mP} = I_P^{(m)} = I_P^m$
(see Zariski-Samuel \cite[Appendix 6, Lemma 5]{ZS}).

An ongoing problem at the intersection of commutative algebra and
algebraic geometry is to study and classify the Hilbert functions
arising from the homogeneous ideals of sets of fat points.  Recall
that if $I \subseteq R = \k[x_0,x_1,x_2]$ is any homogeneous ideal,
then the {\it Hilbert function} of $R/I$, denoted $\H_{R/I}$, is the
numerical function $\H_{R/I}:\mathbb{N} \rightarrow \mathbb{N}$ defined
by
\[\H_{R/I}(t) := \dim_\k R_t - \dim_\k I_t\]
where $R_t$, respectively $I_t$, denotes the $t$-th graded component
of $R$, respectively $I$.  If $I = I_\Z$ is the defining ideal of a
set of (fat) points $\Z$, then we usually write $\H_{\Z}$ for
$\H_{R/I_\Z}$.  The {\it first difference} of the Hilbert function
$\H_{R/I}$, is the function
\[\Delta \H_{R/I}(t) := \H_{R/I}(t) - \H_{R/I}(t-1) ~~\mbox{for all
    $t \geq 0$} ~~\mbox{ where $\H_{R/I}(t) = 0$ for all $t < 0$.}\]

Given a set of points $\Z \subseteq \mathbb{P}^2$, Cooper, Harbourne,
and Teitler \cite{CHT} described a procedure by which one
can find both upper and lower bounds on $\H_\Z(t)$ for all $t \geq 0$.
This procedure, which we describe below, will be instrumental in the
proof of our main results.

Let $\Z=\Z_0$ be a fat point subscheme of
$\mathbb{P}^2$. Choose a sequence of lines $\L_1,\dots,\L_r$ and define
$\Z_i$ to be the residual of $\Z_{i-1}$ with respect to the line
$\L_i$
(i.e. the subscheme of $\mathbb{P}^2$ defined by the ideal $I_{\mathbb{Z}_i} : I_{\L_i}$).   Define the associated {\it reduction vector}
$\mathbf{v} = (v_1,\dots,v_r)$ by taking
$v_i = \deg (\L_i \cap \Z_{i-1})$. In particular, $v_i$ is the sum
of the multiplicities of the points in $\L_i \cap \Z_{i-1}$.
Given $\mathbf{v}=(v_1,\dots,v_r)$, we define
functions
\begin{equation}\label{lowerbound}
  f_{\mathbf{v}} (t) = \sum_{i=0}^{r-1} \min (t-i+1,v_{i+1})
\end{equation}
and
\begin{equation}\label{upperbound}
  F_{\mathbf{v}} (t) = \min_{0\leq i\leq r} \bigg(
    \binom{t+2}{2} - \binom{t-i+2}{2} + \sum_{j=i+1}^{r} v_j
  \bigg).
\end{equation}

\begin{thm}[Cooper-Harbourne-Teitler {\cite[Theorem 1.1]{CHT}}]\label{CHT}
  Let $\Z=\Z_0$ be a fat point scheme in $\mathbb{P}^2$ with reduction
  vector $\mathbf{v} = (v_1,\dots,v_r)$ such that
  $\Z_{r+1} = \varnothing$. Then the Hilbert function $\H_\Z (t)$ of $\Z$ is
  bounded by
  $f_{\mathbf{v}} (t) \leq \H_\Z (t) \leq F_{\mathbf{v}} (t)$.
\end{thm}

\begin{exmp}
  \label{ex:CHT}
  Let $\X$ be the $\k$-configuration of type $(1,3,4,5)$ in Figure
  \ref{fig:CHT}. We illustrate how to use Theorem \ref{CHT} to compute
  $\H_{2\X} (8)$.  We take $\Z = 2\X$, i.e.,\ we assume that each point
  has multiplicity two; this is indicated by the $2$ by each point.

  \begin{figure}[ht]
    \centering

\begin{tikzpicture}[scale=0.36]
  \node at (-2,6) {$\L_4$}; \node at (-2,0) {$\L_3$}; \node at (-1,-2)
  {$\L_2$}; \node at (4,-2) {$\L_1$};
  \clip (-1,-1.2) rectangle (10,8);
  \draw (-1,0)--(10,0); \draw (-1,-1)--(10,10); \draw
  (-2,6)--(13,-3); 
  \draw (2,-6)--(8,12); 
  \node[my circle] at (8,0) [pin={below:2}] {}; \node[my circle] at
  (6.33,1) [pin={above right:2}] {}; \node[my circle] at (3,3)
  [pin={above:2}] {}; \node[my circle] at (1.33,4) [pin={above:2}] {};
  \node[my circle] at (-0.33,5) [pin={above:2}] {}; \node[my square]
  at (0,0) [pin={below:2}] {}; \node[my square] at (1.75,0)
  [pin={below:2}] {}; \node[my square] at (3.5,0) [pin={above:2}] {};
  \node[my square] at (5.5,0) [pin={below:2}] {}; \node[my triangle]
  at (1.5,1.5) [pin={above:2}] {}; \node[my triangle] at (4.5,4.5)
  [pin={above:2}] {}; \node[my triangle] at (6,6) [pin={[pin
    distance=4pt]right:2}] {}; \node[my star] at (5,3) [pin={[pin
    distance=4pt]right:2}] {};
\end{tikzpicture}
\quad
\begin{tikzpicture}[scale=0.36]
  \node at (-2,6) {$\L_4$}; \node at (-2,0) {$\L_3$}; \node at (-1,-2)
  {$\L_2$}; \node at (4,-2) {$\L_1$};
  \clip (-1,-1.2) rectangle (10,8);
  \draw (-1,0)--(10,0); \draw (-1,-1)--(10,10); \draw
  (-2,6)--(13,-3); 
  \draw (2,-6)--(8,12); 
  \node[my circle] at (8,0) [pin={below:1}] {}; \node[my circle] at
  (6.33,1) [pin={above right:1}] {}; \node[my circle] at (3,3)
  [pin={above:1}] {}; \node[my circle] at (1.33,4) [pin={above:1}] {};
  \node[my circle] at (-0.33,5) [pin={above:1}] {}; \node[my square]
  at (0,0) [pin={below:2}] {}; \node[my square] at (1.75,0)
  [pin={below:2}] {}; \node[my square] at (3.5,0) [pin={above:2}] {};
  \node[my square] at (5.5,0) [pin={below:2}] {}; \node[my triangle]
  at (1.5,1.5) [pin={above:2}] {}; \node[my triangle] at (4.5,4.5)
  [pin={above:2}] {}; \node[my triangle] at (6,6) [pin={[pin
    distance=4pt]right:2}] {}; \node[my star] at (5,3) [pin={[pin
    distance=4pt]right:2}] {};
\end{tikzpicture}
\quad
\begin{tikzpicture}[scale=0.36]
  \node at (-2,6) {$\L_4$}; \node at (-2,0) {$\L_3$}; \node at (-1,-2)
  {$\L_2$}; \node at (4,-2) {$\L_1$};
  \clip (-1,-1.2) rectangle (10,8);
  \draw (-1,0)--(10,0); \draw (-1,-1)--(10,10); \draw
  (-2,6)--(13,-3); 
  \draw (2,-6)--(8,12); 
  \node[my circle] at (8,0) [pin={below:0}] {}; \node[my circle] at
  (6.33,1) [pin={above right:1}] {}; \node[my circle] at (3,3)
  [pin={above:1}] {}; \node[my circle] at (1.33,4) [pin={above:1}] {};
  \node[my circle] at (-0.33,5) [pin={above:1}] {}; \node[my square]
  at (0,0) [pin={below:1}] {}; \node[my square] at (1.75,0)
  [pin={below:1}] {}; \node[my square] at (3.5,0) [pin={above:1}] {};
  \node[my square] at (5.5,0) [pin={below:1}] {}; \node[my triangle]
  at (1.5,1.5) [pin={above:2}] {}; \node[my triangle] at (4.5,4.5)
  [pin={above:2}] {}; \node[my triangle] at (6,6) [pin={[pin
    distance=4pt]right:2}] {}; \node[my star] at (5,3) [pin={[pin
    distance=4pt]right:2}] {};
\end{tikzpicture}

\caption{The Cooper-Harbourne-Teitler bound computation.}

\label{fig:CHT}
\end{figure}
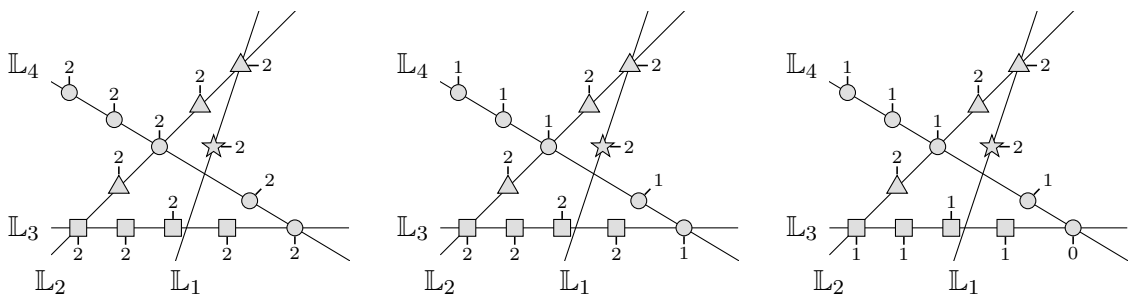

We apply Theorem \ref{CHT} using the sequence of lines
$\HH_1,\dots,\HH_{8}$, where $\HH_1=\L_4$, $\HH_2=\L_3$, $\HH_3=\L_2$,
$\HH_4=\L_1$, $\HH_5=\L_4$, $\HH_6=\L_3$, $\HH_7=\L_2$, and
$\HH_8=\L_1$.  The reduction vector is ${\bf v}=(10,9,8,3,3,3,2,1)$.
To see this, note that $\Z_0 = 2\X$, so $\HH_1 \cap \Z_0$ consists of
the five double points on $\L_4$, so $\deg(\HH_1 \cap \Z_0) = 10$.  We
form $\Z_1$ from $\Z_0$ by reducing the multiplicity of each point on
$\HH_1 = \L_4$ by one.  Then $\HH_2 \cap \Z_1 = \L_3 \cap \Z_1$
consists of 4 double points and one reduced point, so
$\deg(\HH_2 \cap \Z_1) = 2\cdot 4 + 1 = 9$.  Figure \ref{fig:CHT}
illustrates the first two steps of this procedure. Continuing in this
fashion allows us to compute ${\bf v}$, ending when we reach
$\Z_9 = \varnothing$.

To compute the lower bound $f_{\bf v} (8)$  using
Equation \eqref{lowerbound}, we compare the values of $8-i+1$ and
$v_{i+1}$ in the table below (the minimum is in bold).

  \begin{center}
    \begin{tabular}{|c|c|c|c|c|c|c|c|c|c|c|}
      \hline 
      $i$ &   0 & 1 & 2 & 3 & 4 & 5 & 6 & 7  \\ \hline \hline
      $8-i+1$ & \bf 9 & \bf 8 & \bf 7 & 6 & 5 & 4 & 3 & 2 \\ \hline 
      $v_{i+1}$ & 10 & 9 & 8 & \bf 3 & \bf 3 & \bf 3 & \bf 2 & \bf 1 \\ \hline
    \end{tabular}
  \end{center}

Adding up the minimum values gives
$
f_{\bf v}(8)=36.
$
To compute an upper bound, we take $i=3$ in
Equation~\eqref{upperbound}; then we have
$$
F_{\bf v}(8)\leq \binom{8+2}{2}-\binom{5+2}{2}+\sum_{j=4}^8 v_j=36.
$$
This implies that
$
f_{\bf v}(8)=F_{\bf v}(8)=\H_{2\X}(8)=36.
$
\end{exmp}

\begin{rem}
  In Example \ref{ex:CHT}, we have used the procedure of
  \cite{CHT} to find actual values of the Hilbert
  function.  In general, however, one can only expect to find bounds.
\end{rem}

\subsection{Properties of \texorpdfstring{$\k$}{k}-configurations} 
In this section, we record a number of useful facts about $\k$-configurations.

\begin{lem}\label{bounds}
Suppose that $\X \subseteq \mathbb{P}^2$ is a 
$\k$-configuration of type $(d_1,\ldots,d_s)$.  Then
\begin{enumerate}
\item[$(i)$]  $d_j \leq d_s - s + j$ for $j=1,\ldots,s$;
\item[$(ii)$] if $d_s =s$, then $(d_1,\ldots,d_s) = (1,\ldots,s)$;
\item[$(iii)$] for any line $\L$ in $\mathbb{P}^2$, 
$|\L \cap \X| \leq d_s.$
\end{enumerate}
\end{lem}

\begin{proof}
Statements $(i)$ and $(ii)$ follow directly from the definition of
$\k$-configurations since $1 \leq d_1 < d_2 < \cdots < d_s$. 
Statement $(iii)$ is \cite[Lemma 1.3]{RR}.
\end{proof}

By definition, there is at least one line $\L$ in $\mathbb{P}^2$ 
that meets a $\k$-configuration $\X$ of type $(d_1,\ldots,d_s)$ at 
$d_s$ points, namely, the line $\L_s$.  As mentioned in the introduction,
our goal is to enumerate the lines that meet $\X$ at 
exactly $d_s$ points.  We begin with some useful necessary conditions
for a line $\L$ to contain $d_s$ points.

\begin{lem}\label{necessaryconditions}
Suppose that $\X \subseteq \mathbb{P}^2$ is a 
$\k$-configuration of type $(d_1,\ldots,d_s)$, 
and $\L_1,\ldots,\L_s$ are the lines used to define $\X$.
Let $\L$ be any line in $\mathbb{P}^2$ such that $|\L \cap \X| = d_s$.
\begin{enumerate}
\item[$(i)$]  If $d_s > s$, then $\L \in \{\L_1,\ldots,\L_s\}$.
\item[$(ii)$] If $\L = \L_i$, then 
$d_j = d_s-s+j$  for $j = i,\ldots,s$.
\end{enumerate}
\end{lem}

\begin{proof} 
$(i)$ 
If $\L \not\in \{\L_1,\ldots,\L_s\}$, then 
$\L \cap \X \subseteq \bigcup_{i=1}^s (\L \cap \L_i)$.  So
if $s < d_s$,
$|\L \cap \X| \leq \sum_{i=1}^s |\L \cap \L_i| = s < d_s$.  In other words,
if $|\L \cap \X| = d_s$, then $\L$ must be in
$\{\L_1,\ldots,\L_s\}$.
  
$(ii)$ 
Suppose $\L = \L_i$ contains $d_s$ points of $\X$.
By definition, $\L_i$ contains the $d_i$ points of $\X_i \subseteq \X$.
Furthermore, this line cannot contain any of the points
in $\X_1,\ldots,\X_{i-1}$.   In addition, $\L_i$ can contain at
most one point of $\X_{i+1},\ldots,\X_s$.   So
$d_s = |\L_i \cap \X| \leq d_i + (s-i)$.  But by Lemma \ref{bounds},
we have $d_i + (s-i) \leq d_s$, so $d_s = d_i + (s-i)$. 
To complete the proof, note that $d_i < d_{i+1} < \cdots < d_s$ 
is a set of $s-i+1$ strictly increasing integers with 
$d_i = d_s - (s-i)$. This forces $d_j = d_s - (s-j)$ for 
all $j=i,\ldots,s$.
\end{proof}

\begin{rem} If $\X$ is a $\k$-configuration of
type $(d_1,\ldots,d_s)$ with $d_{s-1} < d_{s}-1$, the above lemma
implies that there is exactly one line containing $d_s$
points of $\X$, namely $\L_s$.
\end{rem}

If $d_s > s$, Lemma \ref{necessaryconditions} implies that the lines we want
to count are among the $\L_i$'s, and consequently,
there are at most $s$ such lines.  The next result shows that when $d_s = s$ (or
equivalently, the type is $(1,2,\ldots,s)$) the situation is more subtle.
In particular, if there is a line $\L$ that contains $s$ points that is
not among the $\L_i$'s, then it must be one of two lines.

\begin{lem} \label{extraline}
  Suppose that $\X \subseteq \mathbb{P}^2$ is a
  $\k$-configuration of type $(1,2,\ldots,s)$ with $s \geq 2$. 
  Let
  $\X_1,\ldots,\X_s$ be the subsets of $\X$; let $\L_1,\ldots,\L_s$
  be the lines used to define $\X$; let $\X_1 = \{P\}$ be the point
  on $\L_1$; and let $\X_2 = \{Q_1,Q_2\}$ be the two points on $\L_2$.
  If $\L$ is a line in $\mathbb{P}^2$ such that
  $|\L \cap \X| = d_s = s $, and if $\L \not\in \{\L_1,\ldots,\L_s\}$,
  then $\L$ must either be the line through $P$ and $Q_1$, or the line
  through $P$ and $Q_2$.
\end{lem}

\begin{proof}
  Suppose $|\L \cap \X| = d_s = s$.  Since
  $\L \not\in \{\L_1,\dots,\L_s\}$, we have
  $s = |\L \cap \X| \leq |\L \cap \L_1| + \cdots + |\L \cap \L_s| =
  s$.  In other words, $\L \cap \L_i$ is a point of
  $\X_i \subseteq \X$ for $i=1,\ldots,s$.  So $\L$ must pass through
  $P$ and either $Q_1$ or $Q_2$.
\end{proof}

\begin{cor}\label{numberoflines}
 Suppose that $\X \subseteq \mathbb{P}^2$ is a
  $\k$-configuration of type $(1,2,\ldots,s)$ with $s \geq 2$.  Then
there are at most $s+1$ lines that contain $s$ points of $\X$.
\end{cor}

\begin{proof}
The only candidates for the lines that contain $s$ points are the
$s$ lines $\L_1,\ldots,\L_s$ that define the $\k$-configuration,
and by Lemma \ref{extraline}, the two lines $\L_{PQ_1}$ and $\L_{PQ_2}$,
i.e., the lines that go through the point of $\X_1 = \{P\}$ and one of 
the two points of $\X_2 = \{Q_1,Q_2\}$.  This gives us $s+2$
lines.  However,  if the lines $\L_{PQ_1}$ and $\L_{PQ_2}$ both contain
$s$ points, then either $\L_1$ is one of these two lines, or does
not contain $s$ points.  Indeed if $\L_1$ contains $s$ points,
then $s=|\L_1 \cap \X| = |\X_1| + |\L_1 \cap \L_2| + \cdots + 
|\L_1 \cap \L_s| =s$.  In particular, $|\L_1 \cap \L_2| =1$,
i.e., $\L_1$  must contain one of the 
two points of $\X_2$, and so $\L_1 = \L_{PQ_1}$ or $\L_{PQ_2}$. 
So, there are at most $s+1$ lines that contain $s$ points $\X$.
\end{proof}

We finish this section with a useful lemma for relabelling a
$\k$-configuration.  This lemma exploits the fact that the lines and subsets
defining a $\k$-configuration need not be unique.

\begin{lem}\label{relabel1}
Suppose that $\X \subseteq \mathbb{P}^2$ is a 
$\k$-configuration of type $(d_1,d_2,\ldots,d_s)$ with $s \geq 2$.
Let $\X_1,\ldots,\X_s$ be the subsets of 
$\X$, and $\L_1,\ldots,\L_s$ the lines used to define $\X$.
Suppose that 
\begin{enumerate}
\item[$\bullet$] $|\L_{s-k} \cap \X| = d_s$ for $k = 0,\ldots,j$.
\item[$\bullet$] $|\L_{s-k} \cap \X| < d_s$ for $k = j+1,\ldots,i-1$, and
\item[$\bullet$] $|\L_{s-i} \cap \X| = d_s$.
\end{enumerate}
Set $\mathbb{T} = \L_{s-i} \cap (\X_{s-j-1} \cup \X_{s-j-2} \cup \cdots 
\cup \X_{s-i+1})$.

Then the $\k$-configuration $\X$ can also be defined using
the subsets $\X'_1,\ldots,\X'_s$ and lines $\L'_1,\ldots,\L'_s$
where 
\begin{enumerate}
\item[$\bullet$] $\X'_k = \X_k$ and $\L'_k = \L_k$ for $k=1,\ldots,s-i-1$,
\item[$\bullet$] $\X'_k = \X_{k+1} \setminus \mathbb{T}$ 
and $\L'_k = \L_{k+1}$ for 
$k=s-i,\ldots,s-j-2$,
\item[$\bullet$] $\X'_{s-j-1} = \X_{s-i} \cup \mathbb{T}$ and 
$\L'_{s-j-1} = \L_{s-i}$, and
\item[$\bullet$] $\X'_k = \X_k$ and $\L'_k = \L_k$ for $k=s-j,\ldots,s$.
\end{enumerate}
\end{lem}

\begin{proof}
We need to verify that the subsets $\X'_i$ and lines $\L'_i$
define the same $\k$-configuration, that is, we need to see if they satisfy 
the conditions (1), (2), and (3) of Definition \ref{kdefn}.

We first note that condition $(1)$ holds since
\begin{eqnarray*}
\bigcup_{i=1}^s \X'_k &= & (\X_1 \cup \cdots \cup \X_{s-i-1}) \cup 
\left(\bigcup_{k={s-i}}^{s-j-2} (\X_{k+1} \setminus \mathbb{T})\right) \cup ( \X_{s-i} \cup \mathbb{T}) \cup (\X_{s-j} \cup \cdots \cup \X_s) \\
& = & \X_1 \cup \cdots \cup \X_s = \X.
\end{eqnarray*}

For condition $(2)$, it is clear that $\X'_k \subseteq \L'_k$ for all
$k$.  We now verify that $|\X'_k| = d_k$ for all $k$.  For 
$k=1,\ldots,s-i-1$ and $k = s-j,\ldots,s$ this is immediate
since $\X_k' = \X_k$.    
Because $|\L_{s-i} \cap \X| = d_s$, it follows by Lemma 
\ref{necessaryconditions} that $d_{s-k} = d_s - s + (s-k) = d_s -k$ 
for $k=i,\ldots,j+1$.
Moreover, as in the proof of Lemma \ref{necessaryconditions},
$\L_{s-i} \cap \L_{s-k} \in \X_{s-k}$ for all $k=j+1,\ldots,i-1$.
So $|\mathbb{T}| = i-1-j$, and thus 
\[|\X'_{s-j+1}| = |\X_{s-i} \cup \mathbb{T}|  = d_s - i  + i-1-j = d_s - (j+1) = d_{s-j+1}.\]
Also, again since $\L_{s-i} \cap \L_{s-k} \in \X_{s-k}$ for $k=j+1,\ldots,i-1$,
we have
\[
|\X'_k| = |\X_{k+1} \setminus \mathbb{T}| = d_{k+1} - 1 = d_s-s+k+1 -1 = d_s -s+k
=d_k\]
for $k= s-i,\ldots,s-j-2$. 

Finally, for condition $(3)$, we only need to check the line
$\L'_{s-j-1}$ since the result is true for the other lines by
the construction of $\X$ using the lines $\L_1,\ldots,\L_s$.   
Now $\L'_{s-j-1} = \L_{s-i}$, and we know
that it does not intersect with the points $\X'_k = \X_k$ with
$k < s-i$.  Also, by construction, $\L'_{s-j-1}$ does not
intersect with the points of $\X'_{s-i},\ldots,\X'_{s-j-2}$.  So condition
$(3)$ holds.
\end{proof}

\begin{figure}[ht]
  \centering
  \begin{tikzpicture}[scale=0.40]
    \node at (7,-2) {$\L_1$}; \node at (-1,-2) {$\L_2$}; \node at
    (-2,3) {$\L_3$}; \node at (-2,1.5) {$\L_4$}; \node at (-2,0)
    {$\L_5$};
    \clip (-1,-1) rectangle (11,8); \draw (-1,0)--(11,0); \draw
    (-1,-1)--(10,10); \draw (-1,3)--(11,3); \draw (-1,1.5)--(11,1.5);
    \draw (8,-6)--(2,12); 
    \node[my circle] at (0,0) {}; \node[my circle] at (2,0) {};
    \node[my circle] at (4,0) {}; \node[my circle] at (6,0) {};
    \node[my circle] at (8,0) {}; \node[my circle] at (10,0) {};
    \node[my kite] at (3.25,1.5) {}; \node[my kite] at (5,1.5) {};
    \node[my kite] at (6.75,1.5) {}; \node[my kite] at (8.75,1.5) {};
    \node[my square] at (3,3) {}; \node[my square] at (5,3) {};
    \node[my square] at (7,3) {}; \node[my square] at (9.25,3) {};
    \node[my kite] at (1.5,1.5) {}; \node[my triangle] at (4,4) {};
    \node[my triangle] at (5.5,5.5) {}; \node[my triangle] at (7,7)
    {}; \node[my star] at (3.8,6.6) {};
  \end{tikzpicture}\qquad\qquad
  \begin{tikzpicture}[scale=0.40]
    \node at (7,-2) {$\L'_1$}; \node at (-1,-2) {$\L'_4$}; \node at
    (-2,3) {$\L'_2$}; \node at (-2,1.5) {$\L'_3$}; \node at (-2,0)
    {$\L'_5$};
    \clip (-1,-1) rectangle (11,8); \draw (-1,0)--(11,0); \draw
    (-1,-1)--(10,10); \draw (-1,3)--(11,3); \draw (-1,1.5)--(11,1.5);
    \draw (8,-6)--(2,12); 
    \node[my circle] at (0,0) {}; \node[my circle] at (2,0) {};
    \node[my circle] at (4,0) {}; \node[my circle] at (6,0) {};
    \node[my circle] at (8,0) {}; \node[my circle] at (10,0) {};
    \node[my kite] at (3.25,1.5) {}; \node[my kite] at (5,1.5) {};
    \node[my kite] at (6.75,1.5) {}; \node[my kite] at (8.75,1.5) {};
    \node[my triangle] at (3,3) {}; \node[my square] at (5,3) {};
    \node[my square] at (7,3) {}; \node[my square] at (9.25,3) {};
    \node[my triangle] at (1.5,1.5) {}; \node[my triangle] at (4,4)
    {}; \node[my triangle] at (5.5,5.5) {}; \node[my triangle] at
    (7,7) {}; \node[my star] at (3.8,6.6) {};
  \end{tikzpicture}

  \caption{Relabelling lines of a $\k$-configuration of type 
$(1,3,4,5,6)$}
  \label{fig:relabel}
\end{figure}
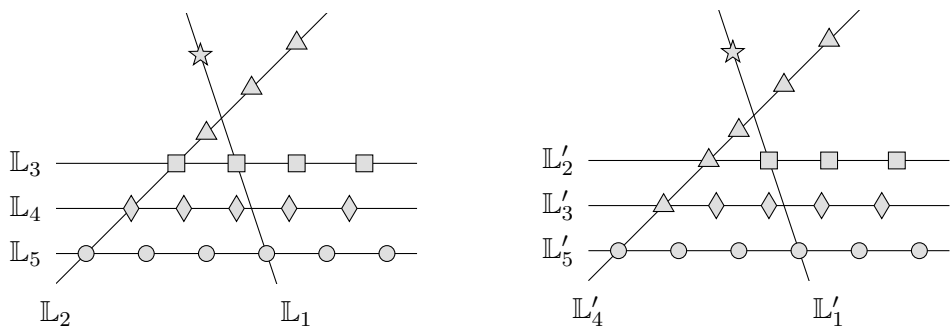

\begin{exmp}
  Figure \ref{fig:relabel} gives an example of the relabelling.  As
  before, the shapes denote which points belong to the subsets $\X_i$.
\end{exmp}

\begin{cor}\label{relabelcor}
Suppose that $\X \subseteq \mathbb{P}^2$ is a 
$\k$-configuration of type $(d_1,d_2,\ldots,d_s)$ with $s \geq 2$.
Let $\L_1,\ldots,\L_s$ be the lines used to define $\X$.  
After relabelling, we can assume that there is an $r$
such that $|\L_{s-j} \cap \X| = d_s$ for all $0 \leq j \leq r-1$,
but $|\L_{s-j} \cap \X| < d_s$ for all $r \leq j \leq s-1$. 
\end{cor}

\begin{proof}
In the assumptions of Lemma \ref{relabel1}, we are assuming
that $\L_s,\ldots,\L_{s-j}$  all
meet $\X$ at $d_s$ points, but $\L_{s-j-1}$ does not.  
After applying the relabelling of Lemma
\ref{relabel1}, the lines $\L'_s,\ldots,\L'_{s-j-1}$ 
now meet $\X$ at $d_s$ points.  By reiterating Lemma \ref{relabel1},
we arrive at the conclusion.
\end{proof}

\begin{rem}\label{hilbertfcnreducedbnd}
In the introduction we mentioned that the number of lines 
that contain $d_s$ points is bounded by a value of the first
difference of the Hilbert function.  Specifically, the number
of lines that contain $d_s$ points is bounded above by 
$\Delta \H_{\X}(d_s-1) + 1$. 

We sketch out how to prove this result.  Roberts and Roitman 
\cite[Theorem 1.2]{RR} give a formula for the Hilbert function
$\H_{\X}$ of a $\k$-configuration in terms of the type 
$(d_1,d_2,\ldots,d_s)$.  It follows
from this formula that $\H_{\X}(d_s-1) = \sum_{i=1}^s d_i$, and 
$\H_{\X}(d_s-2) = (\sum_{i=1}^s d_i) - t$ where $t$ is the number
of consecutive integers at the end of $(d_1,d_2,\ldots,d_s)$.
 So,
\[\Delta \H_\X(d_s-1) = \H_{\X}(d_s-1) - \H_{\X}(d_s-2) = t.\]
Note that if $d_s = s$, then $t=s$.
It
follows by Lemma \ref{necessaryconditions} that if $d_s > s$,
then $t$ is an upper bound on the number of lines that contain $d_s$
points, and if $d_s =s$, then by Corollary \ref{numberoflines},
$t+1 = s+1$ is an upper bound.  We can combine this information
in the statement that  number of lines
that contain $d_s$ points is bounded above by
$\Delta \H_\X(d_s-1) +1$.
\end{rem}


\section{The case \texorpdfstring{$d_s > s$}{ds greater than s}}
\label{sec:main}

In this section, we prove 
Theorem \ref{maintheorem} in the case 
the $\k$-configuration $\X$ has type $(d_1,\dots,d_s)$ with $d_s > s$.

\begin{thm}\label{bigmult}
  Let $\X \subseteq \mathbb{P}^2$ be a $\k$-configuration of type
  $d=(d_1,\dots,d_s)$, and assume that there are $r$ lines containing
  exactly $d_s$ points of $\X$.  If $d_s > s$, then
  $r = \Delta \H_{m\X} (m d_s -1)$ for all $m \geq 2$.
\end{thm}

\begin{proof}
  By Lemma \ref{necessaryconditions} $(i)$, the lines containing $d_s$
  points of $\X$ fall among the lines $\L_1,\dots,\L_s$ defining the
  $\k$-configuration.  By Corollary \ref{relabelcor}, we may assume
  that the lines containing exactly $d_s$ points of $\X$ are
  $\L_s,\dots,\L_{s-r+1}$, while the lines $\L_{s-r},\dots,\L_1$
  contain less than $d_s$ points of $\X$.

  We will apply Theorem \ref{CHT} to compute certain values of
  $\H_{m\X}$. Towards this goal, we obtain the reduction vector
  $\bf v$ of $\X$ using the sequence of lines
  \begin{equation*}
    \L_s,\dots,\L_1,
    \L_s,\dots,\L_1,
    \dots,
    \L_s,\dots,\L_1,
  \end{equation*}
  where the subsequence $\L_s,\dots,\L_1$ is repeated $m$ times.

  We claim that, for $i=1,\dots,r$, we have
  \begin{equation*}
    v_i = m d_s -i+1.
  \end{equation*}
  Let $\Z_0 = m\X$. Since $|\L_s \cap \X| = d_s$, we have
  \begin{equation*}
    v_1 = \deg (\L_s \cap \Z_0) = m d_s.
  \end{equation*}
  Now let $\Z_1$ be the residual of $\Z_0$ with respect to the line
  $\L_s$. The line $\L_{s-1}$ contains the $d_{s-1}$ points of
  $\X_{s-1}$, and the point $\X_s \cap \L_{s-1}$. The multiplicity of
  the point $\X_s \cap \L_{s-1}$ in $\Z_1$ is $m-1$, while the points
  of $\X_{s-1}$ have multiplicity $m$ in $\Z_1$. Thus we get
  \begin{equation*}
    v_2 = \deg (\L_{s-1} \cap \Z_1) = m d_{s-1} +m-1 =
    m (d_s-1) +m-1=m d_s -1,
  \end{equation*}
  where $d_{s-1} = d_s -1$ by Lemma \ref{necessaryconditions}
  $(ii)$. Continuing in this fashion, for $i=3,\dots,r$, we have a
  scheme $\Z_{i-1}$. The line $\L_{s-i+1}$ contains the $d_{s-i+1}$ points
  of $\X_{s-i+1}$, and the points $\X_s \cap \L_{s-i+1}$,
  $\X_{s-1} \cap \L_{s-i+1}$, $\dots$, $\X_{s-i+2} \cap \L_{s-i+1}$. The
  former have multiplicity $m$ in $\Z_{i-1}$, while the latter have
  multiplicity $m-1$ in $\Z_{i-1}$. Thus, for $i=1,\dots,r$, we get
  \begin{equation*}
    \begin{split}
    v_i &= \deg (\L_{s-i+1} \cap \Z_{i-1})
                  = m d_{s-i+1} +(m-1)(i-1)\\
    &= m (d_s-i+1) +(m-1)(i-1)=m d_s -i+1.
  \end{split}
  \end{equation*}

  Next we claim that, for $i=r+1,\dots,s$, we have
  \begin{equation*}
    v_i \leq m d_s -i.
  \end{equation*}
  The line $\L_{s-i+1}$ contains the $d_{s-i+1}$ points of
  $\X_{s-i+1}$, and $e$ points at the intersections
  $\X_t \cap \L_{s-i+1}$ for $t>s-i+1$. Note that
  $d_{s-i+1} + e < d_s$ because we assumed that, for $i=r+1,\dots,s$,
  the line $\L_{s-i+1}$ contains less than $d_s$ points of $\X$.  The points of
  $\X_{s-i+1}$ have multiplicity $m$ in $\Z_{i-1}$, while each point
  $\X_t \cap \L_{s-i+1}$ has multiplicity $m-1$ in $\Z_{i-1}$. Thus,
  for $i=r+1,\dots,s$, we get
  \begin{equation*}
    \begin{split}
    v_i &= \deg (\L_{s-i+1} \cap \Z_{i-1}) = m d_{s-i+1} +(m-1)e\\
              & = d_{s-i+1} +(m-1)(d_{s-i+1} +e)\\
              &< d_s -i+1 + (m-1) d_s = m d_s -i+1,
            \end{split}
  \end{equation*}
  where the inequality uses Lemma \ref{bounds} $(i)$. This proves our
  claim.

  This concludes our first round of removing the lines
  $\L_s,\dots,\L_1$, corresponding to the entries
  $v_1,\dots,v_s$ of the reduction vector ${\bf v}$. Now
  we focus on later passes. We can index later entries of the
  reduction vector by $v_{js+i}$, where $j=1,\dots,m-1$ keeps
  track of the current pass (the first pass corresponding to $j=0$),
  and $i=1,\dots,s$ indicates that we are going to remove the line
  $\L_{s-i+1}$. We claim that
  \begin{equation*}
    v_{js+i} \leq m d_s - (js+i).
  \end{equation*}
  We proceed to estimate the multiplicity of points in
  $\L_{s-i+1} \cap \Z_{js+i-1}$. The line $\L_{s-i+1}$ contains the
  $d_{s-i+1}$ points of $\X_{s-i+1}$; these have multiplicity at most
  $m-j$ in $\Z_{js+i-1}$, because the line $\L_{s-i+1}$ was removed
  $j$ times in previous passes. In addition, the line $\L_{s-i+1}$
  contains $e$ points at the intersections $\X_t \cap \L_{s-i+1}$ for
  $t>s-i+1$, where $d_{s-i+1} + e \leq d_s$ as before. Each of these
  points has been removed $j$ times in previous passes and once in the
  current pass, and therefore, it has multiplicity at most $m-j-1$ in
  $\Z_{js+i-1}$. Altogether, for $j=1,\dots,m-1$ and $i=1,\dots,s$, we
  obtain the following estimate:
  \begin{equation*}
    \begin{split}
    v_{js+i}
    &= \deg (\L_{s-i+1} \cap \Z_{js+i-1}) \leq (m-j) d_{s-i+1} +(m-j-1)e\\
    & = d_{s-i+1} +(m-j-1)(d_{s-i+1} +e)\\
    & \leq d_s -i+1 + (m-j-1) d_s = m d_s - j d_s -i+1\\
    & < m d_s -j s -i +1,
  \end{split}
  \end{equation*}
  using Lemma \ref{bounds} $(i)$ and the hypothesis $d_s > s$. This
  proves our claim.

  Observe that after removing the lines $\L_s,\dots,\L_1$ $m$ times,
  we have $\Z_{ms+1} = \varnothing$. In other words,
  ${\bf v} = (v_1,\dots,v_{ms})$ is a complete reduction
  vector for $m\X$. We can summarize our findings about ${\bf v}$ as
  follows:
  \begin{align*}
    &v_i = md_s -i +1,& &\mbox{for $i=1,\ldots,r$},\\
    &v_i \leq md_s -i,& &\mbox{for $i=r+1,\ldots,ms$}.
  \end{align*}
  
  Now
  we compute the value
  $\H_{m\X} (m d_s -2)$ using Theorem \ref{CHT}.  Recall that a lower bound is given by
  \begin{equation*}
    f_{\bf v} (m d_s - 2) = \sum_{i=0}^{ms-1} \min (m d_s -1-i,v_{i+1}).
  \end{equation*}
  Based on our previous estimates, we have
  \begin{align*}
    &\min (m d_s -1-i,v_{i+1}) = md_s -1-i,&
      &\mbox{for $i=0,\ldots,r-1$},\\
    &\min (m d_s -1-i,v_{i+1}) = v_{i+1},&
      &\mbox{for $i=r,\ldots,ms-1$}.
  \end{align*}
  Hence we get
  \begin{equation*}
    f_{\bf v} (m d_s - 2) = \sum_{i=0}^{r-1} (m d_s -1-i) +
    \sum_{i=r}^{ms-1} v_{i+1}.
  \end{equation*}
  As for the upper bound, it is given by
  \begin{equation*}
    F_{\bf v} (m d_s -2) = \min_{0\leq i\leq ms} \bigg(
      \binom{m d_s}{2} - \binom{m d_s -i}{2} + \sum_{j=i+1}^{ms} v_j
    \bigg).    
  \end{equation*}
  Evaluating the right hand side for $i=r$, we get
  \begin{equation*}
    \begin{split}
      F_{\bf v} (m d_s -2)
      &\leq \binom{m d_s}{2} - \binom{m d_s -r}{2} + \sum_{j=r+1}^{ms} v_j\\
      &=\sum_{h=m d_s -r}^{m d_s -1} h + \sum_{j=r+1}^{ms} v_j\\
      &=\sum_{i=0}^{r-1} (m d_s -1-i) + \sum_{i=r}^{ms-1} v_{i+1}.
    \end{split}
  \end{equation*}
  Combining these bounds, we obtain
  \begin{equation*}
    \H_{m\X} (m d_s -2) =
    \sum_{i=0}^{r-1} (m d_s -1-i) + \sum_{i=r}^{ms-1} v_{i+1}.
  \end{equation*}
  Similarly, we can use Theorem \ref{CHT} to compute $\H_{m\X} (m d_s -1)$. In this case, the lower bound is given by
  \begin{equation*}
    \begin{split}
      f_{\bf v} (m d_s - 1)
      &= \sum_{i=0}^{ms-1} \min (m d_s -i, v_{i+1})\\
      &= \sum_{i=0}^{r-1} (m d_s -i) +
      \sum_{i=r}^{ms-1} v_{i+1} = \sum_{i=0}^{ms-1} v_{i+1}.
    \end{split}
  \end{equation*}
Note that since $f_{\bf v} (m d_s-1)$ is the sum of all
the entries of the reduction vector, $f_{\bf v}(m d_s - 1) = \deg(m\X)$ by
\cite[Remark 1.2.6]{CHT}.  On the other hand, it is well-known
that for any zero-dimensional scheme $\Z$, $\H_{\Z}(t) \leq \deg(\Z)$ for all $t$
(see, e.g. \cite{CTV}).  We thus have
\[ \H_{m\X}(m d_s -1)=\sum_{i=0}^{r-1} (m d_s -i) +
      \sum_{i=r}^{ms-1} v_{i+1}=\deg(m\X).\]
%
%

  Finally, computing the first difference of the Hilbert function gives
the desired result:
  \begin{equation*}
    \begin{split}
      \Delta \H_{m\X} (md_s -1)
      &= \H_{m\X} (md_s -1) - \H_{m\X} (md_s -2) \\
      &= \sum_{i=0}^{r-1} (m d_s -i) - \sum_{i=0}^{r-1} (m d_s -1-i)
      = \sum_{i=0}^{r-1} 1 = r.
    \end{split}
  \end{equation*}
\end{proof}

\begin{exmp}
  Consider the $\k$-configuration $\X$ of type $(1,3,4,5)$
  of Example \ref{ex:CHT}. There
  are three lines containing $d_4=5$ points, namely $\L_2$, $\L_3$, and
  $\L_4$.

  Our computation in Example \ref{ex:CHT} shows that
  $\H_{2\X}(2d_4-2) = \H_{2\X} (8) = 36$. In fact, this is an instance of the general
  computation carried out in the proof of Theorem \ref{bigmult}. A
  similar computation yields $\H_{2\X} (9) = 39$. Therefore we have
$$
\Delta \H_{2\X} (9) = \H_{2\X}(9)-\H_{2\X}(8)=39-36=3,
$$
as desired.
%
\end{exmp}

\section{The case \texorpdfstring{$d_s = s$}{ds=s}.}

In this section we focus on $\k$-configurations of type
$d = (d_1,\ldots,d_s)$ with $d_s = s \geq 2$ (as mentioned in the introduction, the case $d=(1)$ is a single point).  As noted in Lemma \ref{bounds},
the $\k$-configuration $\X$ must have type $(1,2,\ldots,s)$.  
Unlike the case $d_s > s$, the value of $\Delta \H_{2\X}(2d_s-1)$
need not equal the number of lines that contain $d_s = s$
points of $\X$.   As a simple example, 
consider the $\k$-configuration of type $(1,2,3)$
given in Figure \ref{specialcase}.
\begin{figure}[ht]
  \centering
  \begin{tikzpicture}[scale=0.4]
  \node at (-2,6) {$\L_1$};
  \node at (-2,0) {$\L_2$};
  \node at (-1,-2) {$\L_3$};
  \clip (-1,-1) rectangle (10,8);
  \draw (-1,0)--(10,0);
  \draw (-1,-1)--(10,10);
  \draw (-2,6)--(13,-3); 
  \draw[dashed] (2,-6)--(8,12); 
  \node[my circle] at (1,1) {};
  \node[my circle] at (3,3) {};
  \node[my circle] at (6,6) {};
  \node[my square] at (4,0) {};
  \node[my square] at (6.5,0) {};
  \node[my star] at (5.5,1.5) {};
\end{tikzpicture}

\caption{A $\k$-configuration of type $(1,2,3)$ with exactly one line with three
points}

\label{specialcase}
\end{figure}
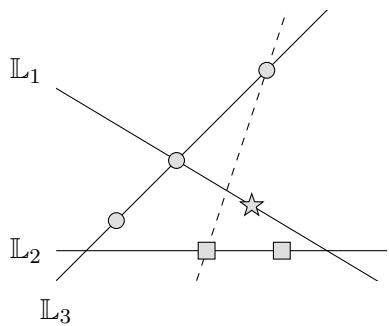
This $\k$-configuration has exactly one line containing exactly three points
(namely, the line $\L_3$).  However, when we compute the Hilbert
function of $2\X$, we get
\[\H_{2\X} : 1~~~3~~~6~~~10~~~~15~~~~18~~~~18~~~~\rightarrow,\]
and consequently, $\Delta \H_{2\X}(2\cdot3-1) = \H_{2\X}(5) - \H_{2\X}(4) = 18-15 = 3$.
So, the hypothesis that $d_s > s$ in Theorem \ref{bigmult} is necessary.   

In this section, we will derive a result similar to Theorem \ref{bigmult}.
However, in order to find the number of lines that contain
$s$ points of $\X$, we need to consider the Hilbert function
of $m\X$ with $m \geq s+1$ instead of $m \geq 2$.  We
need a more subtle argument, in part, because of Lemma \ref{extraline}. 
That is, unlike the case of $d_s > s$, there may be up to two extra lines $\L$
that contains $s$ points of $\X$, where $\L$ is not among the lines
that defines the $\k$-configuration.

We begin with a lemma that allows us to break our argument
into three separate cases.  This lemma is similar to Lemma \ref{relabel1}
in that it allows us to make some additional assumptions about
the lines $\L_1,\ldots,\L_s$ and points $\X_1,\ldots,\X_s$ used
to define the $\k$-configuration.

\begin{lem}\label{relabel2}
Suppose that $\X \subseteq \mathbb{P}^2$ is a 
$\k$-configuration of type $(1,2,\ldots,s)$ with $s \geq 2$.
Let $\X_1,\ldots,\X_s$ be the subsets of 
$\X$, and $\L_1,\ldots,\L_s$ the lines used to define $\X$.

Then one of the three disjoint cases must hold:
\begin{enumerate}
\item[$(i)$] 
  There are exactly $s+1$ lines that contain $s$ points
of $\X$, and the points of $\X$ are precisely the pairwise intersections of such lines.
\item[$(ii)$]  
There are exactly $s$ lines that contain $s$ points of $\X$, and
we can assume that these lines are $\L_1,\ldots,\L_s$. Moreover, for
each $i=1,\dots,s$, the set $\X \cap \L_i$ contains $s-1$ points
located at the intersection of $\L_i$ and $\L_j$ (with $j\neq i$), and a
single point $P_i$ that does not belong to any line $\L_j$ for
$j\neq i$.
\item[$(iii)$] There are $1 \leq r < s$ lines that contain
$s$ points, and furthermore, after a relabelling of the
lines $\L_1,\ldots,\L_s$ and subsets $\X_1,\ldots,\X_s$, we can
assume that none of these $r$ lines pass through
the point of $\X_1 = \{P\}$.
\end{enumerate}
\end{lem}

\begin{proof}

By Corollary \ref{numberoflines} there are 
at most $s+1$ lines that contain $s$ points of $\X$.
So, there are three cases:
$(i)$ exactly $s+1$ lines that contain $s$ point of $\X$, $(ii)$ exactly
$s$ lines that contain $s$ points of $\X$, or $(iii)$ $1\leq r < s$ lines
that contain $s$ points of $\X$.  We now show
that in each case, we can label the $\X_i$'s and $\L_i$'s  as described
in the statement.

$(i)$
  Suppose that there are exactly $s+1$ lines that contain $s$ points
  of $\X$. If $s=2$, then the hypothesis that $\X$ is a
  $\k$-configuration implies that the three points of $\X$ are not
  colinear. Thus each pair of points of $\X$ uniquely determines a
  line, and the points of $\X$ are the intersections of such
  lines. Now suppose that $s>2$, and let $\HH_1,\dots,\HH_s,\HH_{s+1}$
  be the $s+1$ lines each passing through $s$ points of $\X$. Define
  the set of points $\Y := \X \setminus \HH_{s+1}$, so that
  $\X = \Y \cup (\X \cap \HH_{s+1})$. We have
  $|\Y| = \binom{s+1}{2} -s = \binom{s}{2}$. Each line
  $\HH_1,\dots,\HH_s$ passes through $s-1$ points of $\Y$, otherwise
  we would have $|\Y|>\binom{s}{2}$. By induction on $s$, the points
  of $\Y$ are the pairwise intersections of the lines
  $\HH_1,\dots,\HH_s$. By cardinality considerations, the $s$ points
  of $\X \cap \HH_{s+1}$ must be the intersections of $\HH_{s+1}$ with
  the lines $\HH_1,\dots,\HH_s$. This shows that the points of $\X$
  are precisely the pairwise intersections of the lines
  $\HH_1,\dots,\HH_s,\HH_{s+1}$.

$(ii)$
  Suppose that there are exactly $s$ lines that contain $s$ points of $\X$.
There are three subcases: $(a)$ the $s$ lines that contain
$s$ points are $\L_1,\ldots,\L_s$;
$(b)$ $s-1$ of the lines that contain $s$ points are 
among $\L_1,\ldots,\L_s$ and there is one more line $\L$; and 
$(c)$ $s-2$ of the lines that contain $s$ points are among
$\L_1,\ldots,\L_s$, and there are two more lines that contain
$s$ points.  Note that Lemma \ref{extraline} implies that there is at
most two lines not among the $\L_i$'s that will contain $s$ points, so
these are the only three cases.
We will first show that if $(c)$ is true, then we can relabel the lines and
points so that we can assume case $(b)$ is true.  We will then show
that in case $(b)$, we can again relabel lines and points so
we can assume case $(a)$ is true.

Assume case $(c)$ holds.  By Lemma \ref{extraline}, the two
lines that contain $s$ points that are not among the $\L_i$'s are
the lines $\L_{PQ_1}$ and $\L_{PQ_2}$ where
$\X_1 = \{P\}$ and $\X_2 = \{Q_1,Q_2\}$.  As argued in Corollary
\ref{numberoflines}, the line $\L_1$ cannot contain $s$ points.  
Since $\{P\} = \X_1 \subseteq \L_{PQ_1}$,
we then have that the $\k$-configuration can also be defined  by
the same $\X_i$'s and the lines $\L_{PQ_1},\L_2,\ldots,\L_s$.  Note that
we are in now case $(b)$.

We now assume case $(b)$, that is, $s-1$ of the lines that
contain $s$ points are among $\L_1,\ldots,\L_s$ and there is one additional 
line $\L$ that contains $s$ points.   Suppose that $\L_1$ does not
contain $s$ points.  By Lemma \ref{extraline}, the additional
line $\L$ contains $\X_1$, so as above, we replace $\L_1$ with $\L$,
and the $\k$-configuration is defined by the same $\X_i$'s and the lines
$\L,\L_2,\ldots,\L_s$, all of which contain $s$ points.  On the
other hand, suppose $\L_1$ contains $s$ points.
Then there is exactly
one line 
$\L_j \in \{\L_2,\ldots,\L_{s-1}\}$ that does not contain $s$ points (note
that $\L_s$ contains $s$ points).
Moreover,
$\L_1,\ldots,\L_{j-1}$ must all intersect $\L_j$ at distinct
points since each such $\L_i$ needs to contain $s$ distinct points.

Set 
\begin{eqnarray*}
\mathbb{T} & = & \L \cap (\X_1 \cup \cdots \cup\X_j).
\end{eqnarray*}
Since $\L$ contains $s$ points of $\X$, we must have
$\L \cap \X_i \neq \varnothing$ for all $i=1,\ldots,s$,
and in particular, $|\mathbb{T}| = j$.
Then the  $\k$-configuration $\X$ can also be defined using the subsets
\begin{eqnarray*}
 \X'_i &=& (\X_i \setminus \L) \cup (\L_i \cap \L_j) ~~\mbox{and}~~
\L'_i = \L_i ~~\mbox{for $i=1,\ldots,j-1$,} \\
\X_j & = & \mathbb{T} ~~ \mbox{and}~~ \L'_j = \L, ~~\mbox{and}\\
\X'_i & =&\X_i ~~\mbox{and}~~ \L'_i = \L_i ~~\mbox{for $i = j+1,\ldots,s$}.
\end{eqnarray*}
The verification of this fact is similar to the proof of Lemma
\ref{relabel1}.
Note that the line $\L_j$ is no longer used to define the $\k$-configuration;
moreover, the $s$ lines that contain the $s$ points are 
$\L'_1,\ldots,\L'_s$ after this relabelling,  i.e.,
we are now in case $(a)$.   
We have now verified that we can assume that the lines
that contain $s$ points of $\X$ are exactly the lines $\L_1,\ldots,\L_s$.  
We now verify the second part of $(ii)$.

  Now, for each $i=1,\dots,s$, the line $\L_i$ contains exactly $s$
  points of $\X$, so at least one of the $s$ points in $\X \cap \L_i$
  does not belong to $\L_j$ for $j\neq i$; call this point $P_i$. For
  $i=1,\dots,s-1$, set $\Y_i := \X_{i+1} \setminus \{P_{i+1}\}$. The
  set $\Y := \bigcup_{i=1}^{s-1} \Y_i$ is a $\k$-configuration of type
  $(1,\dots,s-1)$ with supporting lines $\L_{i+1} \supseteq \Y_i$
  (this follows from the fact that $\X$ is a $\k$-configuration with
  supporting lines $\L_i \supseteq \X_i$). Furthermore, there are
  exactly $s$ lines that contain $s-1$ points of $\Y$, namely the
  lines $\L_1,\dots,\L_s$. Therefore, by part $(i)$, all points
  of $\Y$ are precisely the pairwise intersections of the lines
  $\L_1,\dots,\L_s$. The statement in part $(ii)$ follows.

$(iii)$  
Finally, suppose that there are $1 \leq r < s$ lines that contain
$s$ points of $\X$.   Like case $(ii)$, there are three subcases:
$(a)$ the $r$ lines are among $\L_1,\ldots,\L_s$; $(b)$ $r-1$ of 
the lines are among $\L_1,\ldots,\L_s$, and there is one 
additional line $\L$, or $(c)$ $r-2$ of the lines are among $\L_1,\ldots,
\L_s$, and there are two additional lines that contains $s$ points.  
Like case $(ii)$, we first show that we can relabel case $(c)$ 
so case $(b)$ is true.  We then show that if case $(b)$ is true,
we can again relabel so case $(a)$ is true.

If we assume case $(c)$, we first apply Corollary \ref{relabelcor}
to relabel the lines so
that $\L_1$ does not contain $s$ points (since only the last $r-2$
lines will contain $s$ points).  Lemma
\ref{extraline} implies that the two additional 
lines are $\L_{PQ_1},\L_{PQ_2}$.  Since $\X_1 \subseteq \L_{PQ_1}$ we
can still define the $\k$-configuration using the same $\X_i$'s, but with
the lines $\L_{PQ_1},\L_2,\ldots,\L_s$,  i.e., we are in case $(b)$.

In case $(b)$, we again first apply Corollary \ref{relabelcor}
to relabel our $\k$-configuration so that $\L_1$ does not contain 
$s$ points.  By Lemma \ref{extraline}, the additional line $\L$
is either $\L_{PQ_1}$ or $\L_{PQ_2}$.  In either case, $\X_1 \subseteq \L$,
so we again define the $\k$-configuration using the the
same $\X_i$'s and the lines $\L,\L_2,\ldots,\L_s$.  We have now
relabelled the $\k$-configuration so case $(a)$ holds.

Since we can assume that $(a)$ holds, the $1 \leq r < s$ lines that
contains $s$ points are among $\L_1,\ldots,\L_s$.  Again, by
applying Corollary \ref{relabelcor}, we can assume that
$\L_{s-r+1},\ldots,\L_s$ are the $r$ lines with $s$ points, and
in particular, none of these points contain $\X_1 = \{P\}$ by definition
of a $\k$-configuration.

\end{proof}

\subsection{Case 1: Exactly \texorpdfstring{$s+1$}{s+1} lines}
We will now consider the three cases of Lemma \ref{relabel2} separately.
We first consider the case that there are exactly $s+1$ lines 
that contain $s$ points of $\X$.

\begin{thm}\label{case=s+1}
  Let $\X \subseteq \mathbb{P}^2$ be a $\k$-configuration of type
  $d=(d_1,\ldots,d_s) = (1,2,\ldots,s)$ with $s \geq 2$.  Assume that there
  are exactly $s+1$ lines containing $s$ points of $\X$. 
  Then
  $s+1 = \Delta \H_{m\X} (m d_s -1)$ for all $m \geq 2$.
\end{thm}

\begin{proof}
  Let $\HH_1,\dots,\HH_s,\HH_{s+1}$ be the lines containing $s$ points
  of $\X$; by Lemma \ref{relabel2} $(i)$, the points of $\X$ are
  precisely the intersections of such lines.

  To compute bounds on the Hilbert function of $m\X$, we apply Theorem
  \ref{CHT} with the reduction vector $\bf v$ obtained from the
  sequence of lines
  \begin{equation*}
    \HH_{s+1}, \HH_s, \dots, \HH_1,
    \HH_{s+1}, \HH_s, \dots, \HH_1,
    \dots,
    \HH_{s+1}, \HH_s, \dots, \HH_1,
  \end{equation*}
  where the subsequence $\HH_{s+1}, \HH_s, \dots, \HH_1$ is repeated
  $\lceil \frac{m}{2}\rceil$ times.

  We index the entries of the reduction vector by $v_{j(s+1)+i}$,
  where $j=0,\dots,\lceil \frac{m}{2}\rceil -1$ is the number of times
  the subsequence of lines $\HH_{s+1}, \HH_s, \dots, \HH_1$ has been
  completely removed, and $i=1,\dots,s+1$ indicates that we are going
  to remove the line $\HH_{s-i+2}$. Note that each time the
  subsequence $\HH_{s+1}, \HH_s, \dots, \HH_1$ is removed, the
  multiplicity of each point of $m\X$ decreases by two. If
  $m$ is even, this process eventually reduces the multiplicity of
  each point to zero. If $m$ is odd, then the process reduces the
  multiplicity of each point to one, so removing the sequence of lines
  $\HH_{s+1}, \HH_s, \dots, \HH_1$ one more time reduces the
  multiplicity to zero. In particular, $m\X$ will be reduced to
  $\varnothing$ after removing the subsequence of lines
  $\HH_{s+1}, \HH_s, \dots, \HH_1$ $\lceil \frac{m}{2}\rceil$
  times. At the step corresponding to $v_{j(s+1)+i}$, the line
  $\HH_{s-i+2}$ contains:
  \begin{itemize}
  \item the points of intersection $\HH_{s-i+2} \cap \HH_k$ for
    $k>s-i+2$, with multiplicity $m-2j-1$;
  \item the points of intersection $\HH_{s-i+2} \cap \HH_k$ for
    $k<s-i+2$, with multiplicity $m-2j$.
  \end{itemize}
  This gives
  \begin{equation}
    \label{eq:star_red_vec}
    \begin{split}
      v_{j(s+1)+i} &= (i-1) (m-2j-1) + (s-i+1) (m-2j)\\
      &= (m-2j)s-i+1.
  \end{split}
  \end{equation}

  When $j=0$, Equation \eqref{eq:star_red_vec} implies
  \begin{equation*}
    v_i = ms -i+1,
  \end{equation*}
  for all $i=1,\dots,s+1$. For $j>0$, we get
  \begin{equation*}
    \begin{split}
      v_{j(s+1)+i} &= (m-2j)s -i+1 = ms -2js -i+1\\
      &< ms -j(s+1) -i+1
    \end{split}
  \end{equation*}
  because $s>1$. This shows that
  \begin{equation*}
    v_{j(s+1)+i} \leq ms -(j(s+1)+i)
  \end{equation*}
  for all $j=1,\dots,\lceil \frac{m}{2}\rceil -1$ and $i=1,\dots,s+1$.

  Since $d_s =s$, we can summarize the results above by writing
  \begin{align*}
      &v_i = m d_s -i +1,& &\mbox{for $i=1,\ldots,s+1$},\\
      &v_i \leq m d_s -i,& &\mbox{for $i=s+2,\ldots,\lceil \frac{m}{2}\rceil (s+1)$}.
  \end{align*}
  Proceeding as in the proof of Theorem \ref{bigmult}, we obtain
  \begin{equation*}
    \H_{m\X} (m d_s -2) = \sum_{i=0}^s (m d_s -1-i) +
    \sum_{i=s+2}^{\lceil \frac{m}{2}\rceil (s+1)-1} v_{i+1}.
  \end{equation*}
  Also as in the proof of Theorem \ref{bigmult}, we have
  \begin{equation*}
    \H_{m\X} (m d_s -1) = \deg (m\X) =
    \sum_{i=0}^{\lceil \frac{m}{2}\rceil (s+1)-1} v_{i+1} =
    \sum_{i=0}^s (m d_s -i) +
    \sum_{i=s+2}^{\lceil \frac{m}{2}\rceil (s+1)-1} v_{i+1}.
  \end{equation*}
  We conclude that
  \begin{equation*}
    \Delta \H_{m\X} (m d_s -1)
    = \H_{m\X} (m d_s -1) - \H_{m\X} (m d_s -2) 
    = s+1.
  \end{equation*}
\end{proof}

\begin{rem}
A $\k$-configuration of type $(1,2,\ldots,s)$ which has exactly $s+1$
lines containing $s$ points is also an example of a star configuration.
When $m=2$, Theorem \ref{case=s+1} can be deduced from \cite[Theorem 3.2]{GHM}.
\end{rem}

\subsection{Case 2: Exactly \texorpdfstring{$s$}{s} lines}
We next consider the case that there are exactly $s$ lines containing
$s$ points. Reasoning as in the previous case, we may compute a
reduction vector from these $s$ lines, in order to calculate values of
the Hilbert function. However, in this case, the bounds thus obtained
may not be tight. The following example illustrates the issue, and a
possible workaround.

\begin{exmp}
  Consider a $\k$-configuration $\X$ of type $(1,2,3,4)$ with exactly
  four lines that contain four points of $\X$. By Lemma
  \ref{relabel2} $(ii)$, $\X$ consists of the intersections of the lines
  $\L_1,\L_2,\L_3,\L_4$ defining the $\k$-configuration, and four non-colinear points $P_1,P_2,P_3,P_4$, with $P_i$ belonging to $\L_i$. We
  have depicted such an $\X$ in Figure \ref{fig:issue_s_lines}.

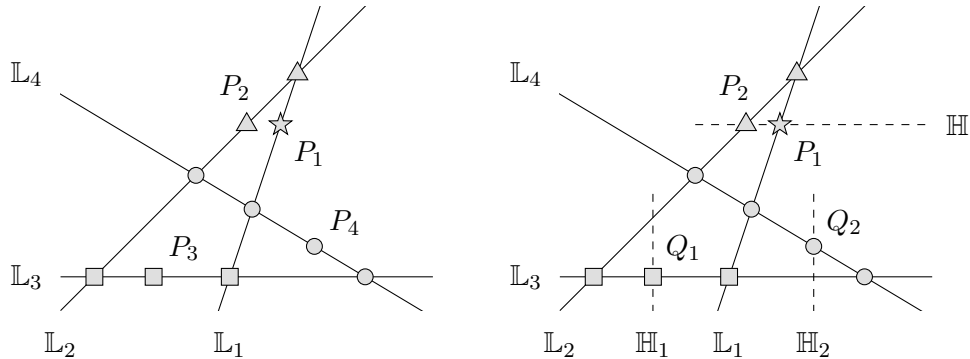
\begin{figure}[ht]
  \centering
    \begin{tikzpicture}[scale=0.45]
      \node at (-2,6) {$\L_4$};
      \node at (-2,0) {$\L_3$};
      \node at (-1,-2) {$\L_2$};
      \node at (4,-2) {$\L_1$};
      \node at (6.35,3.65) {$P_1$};
      \node at (4.15,5.6) {$P_2$};
      \node at (2.65,0.85) {$P_3$};
      \node at (7.4,1.6) {$P_4$};
      \clip (-1,-1) rectangle (10,8);
      \draw (-1,0)--(10,0);
      \draw (-1,-1)--(10,10);
      \draw (-2,6)--(13,-3); 
      \draw (2,-6)--(8,12); 
      \node[my circle] at (8,0) {};
      \node[my circle] at (4.66,2) {};
      \node[my circle] at (3,3) {};
      \node[my circle] at (6.5,0.9) {};
      \node[my square] at (0,0) {};
      \node[my square] at (4,0) {};
      \node[my square] at (1.75,0) {};
      \node[my triangle] at (6,6) {};
      \node[my triangle] at (4.5,4.5) {};
      \node[my star] at (5.5,4.5) {};
    \end{tikzpicture}
    \qquad
    \begin{tikzpicture}[scale=0.45]
      \node at (-2,6) {$\L_4$};
      \node at (-2,0) {$\L_3$};
      \node at (-1,-2) {$\L_2$};
      \node at (4,-2) {$\L_1$};
      \node at (10.75,4.5) {$\HH$};
      \node at (1.75,-2) {$\HH_1$};
      \node at (6.5,-2) {$\HH_2$};
      \node at (6.35,3.65) {$P_1$};
      \node at (4.15,5.6) {$P_2$};
      \node at (2.65,0.85) {$Q_1$};
      \node at (7.4,1.6) {$Q_2$};
      \clip (-1,-1) rectangle (10,8);
      \draw (-1,0)--(10,0);
      \draw (-1,-1)--(10,10);
      \draw (-2,6)--(13,-3); 
      \draw (2,-6)--(8,12); 
      \draw[dashed] (3,4.5)--(12,4.5);
      \draw[dashed] (1.75,-2)--(1.75,2.5);
      \draw[dashed] (6.5,-2)--(6.5,2.5);
      \node[my circle] at (8,0) {};
      \node[my circle] at (4.66,2) {};
      \node[my circle] at (3,3) {};
      \node[my circle] at (6.5,0.9) {};
      \node[my square] at (0,0) {};
      \node[my square] at (4,0) {};
      \node[my square] at (1.75,0) {};
      \node[my triangle] at (6,6) {};
      \node[my triangle] at (4.5,4.5) {};
      \node[my star] at (5.5,4.5) {};
    \end{tikzpicture}
\caption{A $\k$-configuration of type $(1,2,3,4)$ with exactly 4 lines containing 4 points}

\label{fig:issue_s_lines}
\end{figure}

We proceed to compute bounds for $\H_{2\X} (2 d_s -2)$ as we did in
Example \ref{ex:CHT}. First we use the sequence of lines
$\L_4,\L_3,\L_2,\L_1,\L_4,\L_3,\L_2,\L_1$. The table below compares
the function $(2 d_s -2)-i+1$ with the entries of the reduction vector
$\bf v$; the minimum is in bold.
  \begin{center}
    \begin{tabular}{|c|c|c|c|c|c|c|c|c|c|c|}
      \hline 
      $i$ &   0 & 1 & 2 & 3 & 4 & 5 & 6 & 7  \\ \hline \hline
      $6-i+1$ & \bf 7 & \bf 6 & \bf 5 & \bf 4 & 3 & 2 & 1 & \bf 0 \\ \hline 
      $v_{i+1}$ & 8 & 7 & 6 & 5 & \bf 1 & \bf 1 & \bf 1 & 1 \\ \hline
    \end{tabular}
  \end{center}
  Summing the minimum values, we obtain the lower bound
  $\H_{2\X} (6) \geq f_{\bf v} (6) = 25$.

  Now let $\HH$ be the line through $P_1$ and $P_2$.  In general, the
  line $\HH$ could also contain $P_3$ or $P_4$, but not both. However
  in the $\k$-configuration depicted in Figure \ref{fig:issue_s_lines}
  $\HH$ does not contain either $P_3$ or $P_4$.  Consider the points
  of the set $\{P_1,P_2,P_3,P_4\} \setminus \HH$, namely $P_3$ and
  $P_4$, and relabel them $Q_1$ and $Q_2$. Let $\HH_1$ be a line
  through $Q_1$ not passing through $Q_2$, and let $\HH_2$ be a line
  through $Q_2$. We compute a lower bound for $\H_{2\X} (2 d_s -2)$
  using the sequence of lines
  $\L_4,\L_3,\L_2,\L_1,\HH,\HH_1,\HH_2$. The table below summarizes
  the necessary information.
  \begin{center}
    \begin{tabular}{|c|c|c|c|c|c|c|c|c|c|}
      \hline 
      $i$ &   0 & 1 & 2 & 3 & 4 & 5 & 6  \\ \hline \hline
      $6-i+1$ & \bf 7 & \bf 6 & \bf 5 & \bf 4 & 3 & 2 & 1 \\ \hline 
      $v_{i+1}$ & 8 & 7 & 6 & 5 & \bf 2 & \bf 1 & \bf 1 \\ \hline
    \end{tabular}
  \end{center}
  Summing the minimum values, we obtain the lower bound
  $\H_{2\X} (6) \geq f_{\bf v} (6) = 26$.

  An easy computation with Equation \eqref{upperbound} (using either
  reduction vector) leads to the upper bound
  $\H_{2\X} (6) \leq F_{\bf v} (6) \leq 26$. This shows that
  $\H_{2\X} (6) = 26$. In particular, the lower bound computed from
  the lines $\L_4,\L_3,\L_2,\L_1$ alone is not tight.
\end{exmp}

Using the above example as a guide, we prove our main result for the
case under consideration.

\begin{thm}\label{case=s}
  Let $\X \subseteq \mathbb{P}^2$ be a $\k$-configuration of type
  $d=(d_1,\ldots,d_s) = (1,2,\ldots,s)$ with $s\geq 2$.  Assume that there
  are exactly $s$ lines containing $s$ points of $\X$. 
  Then
  $s = \Delta \H_{m\X} (m d_s -1)$ for all $m \geq 2$.
\end{thm}

\begin{proof}
  By Lemma \ref{relabel2} $(ii)$, we can assume that the lines
  containing $s$ points of $\X$ are the lines $\L_1,\ldots,\L_s$ that
  define the $\k$-configuration. Moreover, for each $i=1,\dots,s$,
  there is a point $P_i \in \X\cap \L_i$ that does not belong to
  $\L_j$ for any $j\neq i$. Then the points of $\X$ are the points of
  intersection of the lines $\L_1,\dots,\L_s$ together with the points
  $P_1,\dots,P_s$.

  To compute bounds on the Hilbert function of $m\X$, we apply Theorem
  \ref{CHT} with the reduction vector $\bf v$ obtained from a
  sequence of lines
  \begin{equation*}
    \L_s,\dots,\L_1,
    \L_s,\dots,\L_1,
    \dots,
    \L_s,\dots,\L_1,
    \HH, \HH_1, \dots, \HH_{s-u},
  \end{equation*}
  where the subsequence $\L_s,\dots,\L_1$ is repeated $m-1$ times and
  the additional lines $\HH, \HH_1, \dots, \HH_{s-u}$ are constructed
  as follows.

  Let $\HH$ denote the line through $P_1$ and $P_2$. The line $\HH$
  contains $u$ points of the set $\{P_1,\dots,P_s\}$, where, by
  construction, $u\geq 2$. Furthermore, $\HH$ is not one of the lines
  $\L_1,\dots,\L_s$, and therefore, it cannot contain $s$ points of $\X$,
  i.e.,\ $u\leq s-1$. It follows that the set
  $\{P_1,\dots,P_s\} \setminus \HH$ is not empty, and must in fact
  contain $s-u$ points, which we denote $Q_1,\dots,Q_{s-u}$. For each
  $i=1,\dots,s-u$, let $\HH_i$ be a line passing through $Q_i$ that
  does not contain any point $Q_j$ for $j>i$. 

  Now we proceed to compute (or bound) the entries of the reduction
  vector $\bf v$.  We claim that, for $i=1,\dots,s$, we have
  \begin{equation*}
    v_i = m s -i+1.
  \end{equation*}
  At the $i$-th step, the line $\L_{s-i+1}$ contains:
  \begin{itemize}
  \item the points of intersection $\L_{s-i+1} \cap \L_k$ for $k>s-i+1$, with multiplicity $m-1$;
  \item the points of intersection $\L_{s-i+1} \cap \L_k$ for $k<s-i+1$, with multiplicity $m$; and
  \item the point $P_{s-i+1}$ with multiplicity $m$.
  \end{itemize}
  This gives
  \begin{equation*}
    v_i = (m-1)(i-1) + m (s-i) + m = ms -i+1,
  \end{equation*}
  proving the claim.

  Next, we claim that, for $l=s+1,\dots,(m-1)s+s-u+1$, we have
  \begin{equation}
    \label{eq:claim}
    v_l \leq m s - l.
  \end{equation}

  We first prove this claim for entries $v_{js+i}$, where
  $j=1,\dots,\lceil \frac{m}{2} \rceil -1$ is the number of times the
  subsequence of lines $\L_s,\dots,\L_1$ has been completely removed,
  and $i=1,\dots,s$ indicates that we are going to remove the line
  $\L_{s-i+1}$. At the step corresponding to $v_{js+i}$, the
  line $\L_{s-i+1}$ contains:
  \begin{itemize}
  \item the points of intersection $\L_{s-i+1} \cap \L_k$ for
    $k>s-i+1$, with multiplicity $m-2j-1$;
  \item the points of intersection $\L_{s-i+1} \cap \L_k$ for
    $k<s-i+1$, with multiplicity $m-2j$; and
  \item the point $P_{s-i+1}$ with multiplicity $m-j$.
  \end{itemize}
  The $2j$ in the above multiplicities follows from the fact that
  points located at the intersections of the lines $\L_s,\dots,\L_1$
  are removed twice with each full pass along the subsequence
  $\L_s,\dots,\L_1$.  Thus we obtain
  \begin{equation*}
    \begin{split}
      v_{js+i}
      &= (m-2j-1)(i-1)+(m-2j)(s-i) +m-j\\
      &= ms -i+j-2js+1 = ms -js-i + j(1-s) +1\\
      &< ms -js -i +1,
    \end{split}
  \end{equation*}
  from which the claim of Equation \eqref{eq:claim} follows for the
  chosen values of $i$ and $j$.

  Next, we prove the claim for $v_{js+i}$, where
  $j=\lceil \frac{m}{2} \rceil, \dots, m-2$, and $i=1,\dots,s$. Since
  the lines $\L_s,\dots,\L_1$ have been removed
  $\lceil \frac{m}{2} \rceil$ times, the multiplicity of the points
  located at the intersections of the lines $\L_s,\dots,\L_1$ is now
  zero. Hence, at the step corresponding to $v_{js+i}$, the line
  $\L_{s-i+1}$ only contains the point $P_{s-i+1}$ with multiplicity
  $m-j$.  We get
  \begin{equation*}
    v_{js+i}
    = m-j. 
  \end{equation*}
  Since $j\leq m-2$, we have $m-j\geq 2$ and therefore
  \begin{equation*}
    \frac{m-j}{m-j-1} = \frac{m-j-1+1}{m-j-1} =
    1 + \frac{1}{m-j-1} \leq 2 \leq s.
  \end{equation*}
  This implies
  \begin{equation*}
    v_{js+i} \leq (m -j -1)s = ms -js -s \leq ms -js-i,
  \end{equation*}
  thus proving the claim of Equation \eqref{eq:claim} for the given
  $i$ and $j$.

  At this stage, the multiplicity of the points $P_1,\dots,P_s$ has
  been reduced to one, because each line $\L_1,\dots,\L_s$ has been
  removed $m-1$ times. The next step is to find the value of
  $v_{(m-1)s +1}$, which corresponds to the line $\HH$ defined
  at the beginning. By construction, $\HH$ contains $u$ points of the
  set $\{P_1,\dots,P_s\}$, with $u \leq s-1$. Therefore
  \begin{equation*}
    v_{(m-1)s +1} = u \leq s-1 = ms -((m-1)s+1);
  \end{equation*}
  this shows that Equation \eqref{eq:claim} holds for this entry of $\bf v$.

  Finally, we evaluate $v_{(m-1)s+h}$, for $h=2,\dots,s-u+1$. For a
  given value of $h$, we assume that we have already removed
  $\HH_1,\dots,\HH_{h-2}$ and we are about to remove $\HH_{h-1}$. The
  line $\HH_{h-1}$ contains a single point of $\X$, namely
  $Q_{h-1}$. Moreover, $Q_{h-1}$ is by definition one of the points in
  the set $\{P_1,\dots,P_s\}\setminus \HH$, so its multiplicity is
  down to one. Thus we have
  \begin{equation*}
    v_{(m-1)s+h} = 1 \leq s-h = ms - ((m-1)s +h).
  \end{equation*}
  The inequality $1\leq s-h$ follows from $h\leq s-u+1$ and $u\geq
  2$. Thus we have proved that Equation \eqref{eq:claim} holds for all
  the desired values.

  To summarize, we showed that
  \begin{align*}
      &v_i = ms -i +1,& &\mbox{for $i=1,\ldots,s$},\\
      &v_i \leq ms -i,& &\mbox{for $i=s+1,\ldots,(m-1)s+s-u+1$}.
  \end{align*}
  From here on, the proof proceeds as for Theorem \ref{bigmult}, yielding
  \begin{equation*}
    \begin{split}
      \Delta \H_{m\X} (ms -1)
      &= \H_{m\X} (ms -1) - \H_{m\X} (ms -2) \\
      &= \sum_{i=0}^{s-1} (ms -i) - \sum_{i=0}^{s-1} (ms -1-i)
      = \sum_{i=0}^{s-1} 1 = s.
    \end{split}
  \end{equation*}
\end{proof}

\subsection{Case 3: \texorpdfstring{$1 \leq r < s$}{1=<r<s} lines}
We consider the final case when there are $1 \leq r < s$ lines
that contain $s$ points of $\X$. 
Before going forward, we recall
a result of Catalisano, Trung, and Valla \cite[Lemma 3]{CTV}; we have
specialized this result to the case of points in $\mathbb{P}^2$.

\begin{lem}\label{CTVlemma}
Let $P_1,\ldots,P_k, P$ be distinct points in $\mathbb{P}^2$ and
let $I_P$ be the defining prime ideal of $P$.  
If $m_1,\ldots,m_k,$ and $a$ are positive integers and 
$I = I_{P_1}^{m_1} \cap \cdots \cap I_{P_k}^{m_k}$, then
\begin{enumerate}
\item[$(a)$] $\H_{R/(I+I_P^a)}(t) = 
\sum_{i=0}^{a-1} \dim_\k [(I+I_P^i)/(I+I_P^{i+1})]_t$ 
for every $t>0$, with $I^0_{P}=R$. 
\item[$(b)$] If $P = [1:0:0]$, then 
$[(I+I_P^i)/(I + I_P^{i+1})]_t = 0$ if and only  if $i > t$
or $x_0^{t-i}M \in I + I_P^{i+1}$ for every monomial $M$ of degree $i$ 
in $x_1,x_2$.
\end{enumerate}
\end{lem}

We now prove the remaining open case.  Note that unlike Theorems
\ref{case=s+1} and \ref{case=s}, we need to assume that $m \geq s+1$
instead of $m \geq 2$.

\begin{thm}\label{case<s}
  Let $\X \subseteq \mathbb{P}^2$ be a $\k$-configuration of type
  $d=(d_1,\ldots,d_s) = (1,2,\ldots,s)$ with $s\geq 2$.  Assume that there
  are $1 \leq r < s$  lines containing $s$ points of $\X$. 
  Then
  $r = \Delta \H_{m\X} (m d_s -1)$ for all $m \geq s+1$.
\end{thm}

\begin{proof}
  Let $\L_1,\ldots,\L_s$ be the $s$ lines that define the
  $\k$-configuration.  After using Lemma \ref{relabel2} $(iii)$ to
  relabel, we can assume that the $r$ lines that contain $s$ points
  are among $\L_2,\ldots,\L_s$, and thus the unique point $P$ of
  $\X_1$ does not lie on any line containing $s$ points of $\X$.  So,
  we can write our $\k$-configuration as $\X = \Y \cup \{P\}$ where
  $\{P \} = \X_1$ is the point on the line $\L_1$ and $\Y$ is a
  $\k$-configuration of type
  $(2,3,\ldots,s) = (d_1',\ldots,d_{s-1}')$.  Since there is no line
  containing $s$ points of $\X$ that passes through the point $P$,
  the $r$ lines that contain $s$ points of $\X$ must also contain $s$
  points of $\Y$.  So we can apply Theorem \ref{bigmult} to $\Y$.

  Suppose that $\X_2 = \{Q_1,Q_2\}$.  Since $P,Q_1,Q_2$ do not all
  lie on the same line, we can make a linear change of coordinates so
  that
  \[P = [1:0:0],~~~ Q_1 = [0:1:0],~~~~\mbox{and}~~~~ Q_2 = [0:0:1].\]
  We let $L_i$ denote the linear form that defines the line $\L_i$.
  Note that after we have made our change of coordinates, if $\L$, with
  defining form $L = ax_0 + bx_1 + cx_2$, is any line that does
  not pass through $P$, then
  $L \not\in I_P = \langle x_1,x_2 \rangle$, i.e., $a \neq 0$.

  With this setup, we make the following claim:

\noindent
{\it Claim.}  For all $m \geq s+1$,  $\H_{R/(I_{m\Y}+I_P^m)}(ms-2) = 0.$
\vspace{.25cm}

\noindent
{\it Proof of the Claim.}
By Lemma \ref{CTVlemma}, it suffices to show that for each $i=0,\ldots,m-1$,
the monomial $x_0^{ms-2-i}M  \in \big[I_{m\Y}+I_P^{i+1}\big]_{ms-2}$ 
where $M$ is any
monomial of degree $i$ in $x_1,x_2$.   We will treat the
cases $i \in \{0,\ldots,m-2\}$ and $i = m-1$ separately.

Fix an $i \in \{0,\ldots,m-2\}$ and let $M$ be any monomial
of degree $i$ in $x_1,x_2$.   
Since none of the lines $\L_2,\ldots,\L_s$ pass through the point
$P$, we have $L_k = a_{k,0}x_0 + a_{k,1}x_1 + a_{k,2}x_2$ with $a_{k,0} \neq 0$
for all $k=2,\ldots,s$.
Then 
\[L_2^mL_3^m \cdots L_s^m = ax_0^{ms-m} + \sum_{k=1}^{ms-m} x_0^{ms-m-k}f_k(x_1,x_2)\]
with $a \neq 0$ and where 
$f_k(x_1,x_2)$ is a homogeneous polynomial of degree $k$
only in $x_1$ and $x_2$.  Since $L_2\cdots L_s \in I_\Y$, it follows
that 
\[
L_2^m \cdots L_s^m \in \big[(I_\Y)^m\big]_{ms-m} \subseteq  \big[I_{m\Y}\big]_{ms-m} \subseteq \big[I_{m\Y} + I_P^{i+1}\big]_{ms-m}
\]
and thus
\[
ax_0^{ms-m}M + \sum_{k=1}^{ms-m} x_0^{ms-m-k}f_k(x_1,x_2)M \in \big[I_{m\Y} + I_P^{i+1}\big]_{ms-m+i}.  
\] 
But $I_P^{i+1} = \langle x_1,x_2 \rangle^{i+1}$, so $f_k(x_1,x_2)M \in I_P^{i+1}$
for each $k=1,\ldots,ms-m$ since $f_k(x_1,x_2)M$ is a homogeneous
polynomial only in $x_1,x_2$ of degree $i+k \geq i+1$.
But
then this means that 
\[
a^{-1}ax_0^{ms-m}M = x_0^{ms-m}M\in \big[I_{m\Y} + I_P^{i+1}\big]_{ms-m+i}.
\]  
Since $i \leq m-2$, we thus have $x_0^{m-2-i}x_0^{ms-m}M = x_0^{ms-2-i}M \in
 \big[I_{m\Y} + I_P^{i+1}\big]_{ms-2}$.

Now suppose that $i=m-1$.   Consider any monomial $M = x_1^ax_2^b$ with
$a+b = m-1$ and $a,b \geq 1$.  Since $I_{Q_1} = \langle x_0,x_2 \rangle$ 
and $I_{Q_2} = \langle x_0,x_1 \rangle$, this means that $x_1^ax_2^b \in
I_{Q_1}^b \cap I_{Q_2}^a$.  Because $\L_2$ is the line that 
passes through $Q_1$ and $Q_2$, we have 
$L_2^{m-1}M \in (I_{Q_1}^m \cap I_{Q_2}^m)$, and consequently,
\[
\begin{array}{llllllllll}
L_2^{m-1} L_3^m\cdots L_s^mM 
& = & \ds ax_0^{ms-m-1}M + \sum_{k=1}^{ms-m-1} x_0^{ms-m-1-k}f_k(x_1,x_2)M \\[2.5ex] 
& \in & \big[I_{m\Y} \big]_{ms-2} \subseteq \big[I_{m\Y} + 
I_P^{m}\big]_{ms-2}.  
\end{array}
\] 
Arguing as above, this implies that $x_0^{ms-m-1}M \in \big[I_{m\Y} + 
I_P^{m}\big]_{ms-2}.$

It remains to show that $x_0^{ms-m-1}x_1^{m-1}$ and $x_0^{ms-m-1}x_2^{m-1}
\in  \big[I_{m\Y} + I_P^{m}\big]_{ms-2}.$  We only verify the second statement since the
first statement is similar.  Consider the line $\L$ through the
point $P$ and $Q_2$.  Because $\L$ goes through $P$, it does not
contain $s$ points.  In particular, there must be some
$j \in \{3,\ldots,s\}$ such that $\L \cap \X_j = \varnothing$, i.e.,
$\L$ does not intersect with any of the points of $\X$ on the line
$\L_j$.   Let $\X_j = \{S_1,\ldots,S_j\}$ be these $j$ points,
and let $\mathbb{H}_\ell$ be the line through $Q_2$ and $S_\ell$ for 
$\ell = 1,\ldots,j$.  Furthermore, let $H_\ell$ denote the associated
linear form.
Note that none of the lines $\mathbb{H}_\ell$ can pass
through the point $P$, so in particular, each $H_\ell$ has 
the form $H_\ell = a_{\ell}x_0 + b_{\ell}x_1 + c_{\ell}x_2$ with 
$a_{\ell} \neq 0$.

We now claim that 
\[F := x_1^{m-1}
H_{1}\cdots H_{j} L_2^{m-j}L_3^{m}\cdots L_{j-1}^mL_{j}^{m-1}
L_{j+1}^m\cdots L_s^m \in I_{m\Y}.\]   
Because $j \leq s$ and $m \geq s+1$, $m-j \geq 1$.  So, in
particular, $x_1^{m-1} L_2^{m-j} \in I_{Q_1}^m$.  Also, 
$H_{1}\cdots H_{j} L_2^{m-j} \in I_{Q_2}^m$,
so $F$ vanishes at the points of $\{Q_1,Q_2\}$
to the correct multiplicity.  Note that 
$H_{1}\cdots H_{j} L_{j}^{m-1}$ vanishes at 
all the points on $\L_j$ to multiplicity at least $m$.  Furthermore,
for any other $k$, $L_k^m$ vanishes at all the points on $\L_k$
to multiplicity at least $m$.  So we have $F \in I_{m\Y}$.  
To finish the proof, we need to note that 
\[H_{1}\cdots H_{j} L_2^{m-j}L_3^{m}\cdots L_{j-1}^mL_{j}^{m-1}
L_{j+1}^m\cdots L_s^m = 
ax_0^{ms-m-1} + \sum_{k=1}^{ms-m-1} x_0^{ms-m-1-k}f_k(x_1,x_2)\]
with $a \neq 0$ and where each $f_k(x_1,x_2)$ is a homogeneous polynomial
of degree $k$ only in $x_1,x_2$.  The rest of the proof now follows
similar to the cases above. This ends the proof of the claim.

We now complete the proof.
Let $m \geq s+1$ be any integer, and consider the short exact
sequence
\[0 \rightarrow (I_{m\Y} \cap I_P^m) \rightarrow
I_{m\Y} \oplus I_P^m \rightarrow I_{m\Y}+I_P^m \rightarrow 0.\]
Note that the ideal of $m\X$ is $I_{m\X} = I_{m\Y} \cap I_P^m$, so the
short exact sequence implies
\[\H_{m\X}(t) = \H_{m\Y}(t) + \H_{mP}(t) - \H_{R/(I_{m\Y}+I_P^m)}(t)\] 
for all $t \geq 0$.  Note that $\H_{mP}(t) = \binom{m+1}{2}$ for all 
$t \geq m-1$.   Using this fact, and the above claim we get
\begin{eqnarray*}
\Delta \H_{m\X}(ms-1) &= &\H_{m\X}(ms-1) - \H_{m\X}(ms-2) \\
& = &\left(\H_{m\Y}(ms-1) + \H_{mP}(ms-1) - \H_{R/(I_{m\Y}+I_P^m)}(ms-1)\right) - \\
&&
\left(\H_{m\Y}(ms-2) + \H_{mP}(ms-2) - \H_{R/(I_{m\Y}+I_P^m)}(ms-2)\right) \\
& = & \Delta \H_{m\Y}(ms-1) + \left(\binom{m+1}{2}-\binom{m+1}{2}\right)
- (0 -0) \\
& = & r.
\end{eqnarray*} 
The last equality comes from Theorem \ref{bigmult} since 
$\Delta \H_{m\Y}(ms-1) = r$ for all $m \geq 2$.
\end{proof}

\begin{rem} Notice that in the proof of Theorem \ref{case<s}, the hypothesis
that $m \geq s+1$ was only used in the proof of the claim to show that a
particular monomial belonged to the ideal $I_{m\Y}+I_P^m$.  However, there
may be some room for improvement on the lower bound $s+1$.  For example,
for the $\k$-configuration of type $(1,2,3)$ 
 given in Figure \ref{specialcase}, computer
tests have shown that $\Delta \H_{m\X}(m3-1) = 1$ for all $m \geq s =3$,
instead of $s+1 = 4$.  Similarly, if we consider {\it standard 
linear configurations} of type $(1,2,\ldots,s)$
(as defined in \cite[Definition 2.10]{GMS}), then it can be shown
that Theorem \ref{case<s} holds for all $m \geq 2$.  We omit this
proof since it requires the special geometry of standard linear
$\k$-configurations.
\end{rem}


\section{Concluding Remarks}

We conclude this paper with some observations.   Following \cite{CTV},
we define the {\it regularity index} of a zero-dimensional scheme
$\Z \subseteq \mathbb{P}^n$ to be
\[{\rm ri}(\Z) = \min (t ~|~ \H_{\Z}(t) = \deg(\Z)).\]
Embedded in our proof of Theorem \ref{maintheorem}, we actually
computed the regularity index of multiples of a $\k$-configuration.  In particular,
we proved that

\begin{cor} 
 Let $\X \subseteq \mathbb{P}^2$ be a $\k$-configuration of type
  $d=(d_1,\dots,d_s)$.  Then for all integers $m \geq s+1$,
\[{\rm ri}(m\X) = md_s -1. \]
\end{cor}

\begin{proof}
  If $d=(1)$, then $\X =\{P\}$ is a single point. It is well-known that
  \begin{equation*}
    \H_{m \X} (t) = \min \left(
      \binom{t+2}{2}, \binom{m+1}{2}
    \right),
  \end{equation*}
  so the regularity index is $m-1$.
    
  If $d\neq (1)$ and if $m \geq s+1$,
Theorems \ref{bigmult}, \ref{case=s+1},
\ref{case=s}, and \ref{case<s}, imply $\H_{m\X}(md_s-2) < \H_{m\X}(md_s-1)$.
Moreover, as part of our proofs, we argued that 
$\H_{m\X}(md_s-1) = \deg(m\X)$.
\end{proof}

The regularity index ${\rm ri}(\Z)$ can also be defined as the 
maximal integer $t$ such that $\Delta \H_\Z(t) \ne 0$.
So, Theorem \ref{maintheorem} can be restated as:

\begin{thm}\label{restatement}
  Let $\X \subseteq \mathbb{P}^2$ be a $\k$-configuration of type
  $d=(d_1,\dots,d_s) \neq (1)$ and $m \geq s+1$.   Then  the number
  of lines containing exactly $d_s$ points of $\X$ is the last
  non-zero value of $\Delta \H_{m\X}(t)$.
\end{thm}

As a final comment, we turn to a question posed
by Geramita, Migliore, and Sabourin \cite{GMS}:

\begin{ques}\label{question}
What are all the possible Hilbert functions of fat point schemes
in $\mathbb{P}^n$ whose support has a fixed Hilbert function $\H$?
\end{ques}

As noted in \cite{GMS}, this question is quite difficult;  in fact,
\cite{GMS} focused on the case of double points in $\mathbb{P}^2$.
Using the work of this paper, we can give an interesting
observation related Question \ref{question}:

\begin{thm} Fix integers $m \geq s+1 \geq 3$.
Then there are at least $s+1$ possible Hilbert functions
of homogeneous fat points of multiplicity $m$ in $\mathbb{P}^2$ whose
support has the Hilbert function
\[\H(t) = \min\left( \binom{t+2}{2}, \binom{s+1}{2} \right).\]
\end{thm}

\begin{proof}
  Any $\k$-configuration $\X$ of type $(1,2,\ldots,s)$ with $s \geq 2$
  has Hilbert function $\H_\X(t) = \H(t)$ (see \cite[Theorem
  1.2]{RR}).  By Theorem \ref{restatement},
  $\H_{m\X}(ms-2) = \deg(m\X) - r$ where $r$ is the number of lines
  that contain $s$ points of $\X$.  As shown in Lemma \ref{relabel2},
  $1 \leq r \leq s+1$.  It suffices to show that each $r$ is possible;
  this would imply that we have at least $s+1$ different Hilbert
  functions.

Fix $\L_1,\ldots,\L_{s+1}$ distinct lines.  If $r = s+1$, we take
all pairwise intersections of these $s+1$ lines to get the desired
set of points.  So, suppose $1 \leq r \leq s$.
We construct a $\k$-configuration of type $(1,2,\ldots,s)$ with
exactly $r$ lines containing $s$ points as follows:
\begin{enumerate}
\item[$\bullet$] 
let $\X_1$ be any point in $\L_1 \setminus (\L_2\cup \cdots
\cup \L_{s+1})$.
\item[$\bullet$] let $\X_2$ be any two points $\L_2 \setminus (\L_1 \cup \L_3
\cup \cdots \cup \L_{s+1})$.
\item[$\vdots$]
\item[$\bullet$] let $\X_{s-r}$ be any $s-r$ points $\L_{s-r} \setminus 
(\L_1 \cup \cdots \cup \widehat{\L}_{s-r} \cup \cdots \cup \L_s)$.
\item[$\bullet$] let $\X_{s-r+1}$ be any $s-r+1$ points $\L_{s-r+1} \setminus 
(\L_1 \cup \cdots \cup \widehat{\L}_{s-r+1} \cup \cdots \cup \L_s)$.
\item[$\bullet$] let $\X_{s-r+2}$ any $s-r+1$ points on
 $\L_{s-r+2} \setminus 
(\L_1 \cup \cdots \cup \widehat{\L}_{s-r+2} \cup \cdots \cup \L_s)$
and the point $\L_{s-r+2} \cap \L_{s-r+1}$.
\item[$\bullet$] let $\X_{s-r+3}$ be any $s-r+1$ points on
 $\L_{s-r+3} \setminus 
(\L_1 \cup \cdots \cup \widehat{\L}_{s-r+3} \cup \cdots \cup \L_s)$
and the two points $\L_{s-r+3} \cap (\L_{s-r+1} \cup \L_{s-r+2})$.
\item[$\vdots$]
\item[$\bullet$] let $\X_{s}$ be any $s-r+1$ points on
 $\L_{s} \setminus 
(\L_1 \cup \cdots \cup \widehat{\L}_{s})$
and the $r-1$ points $\L_{s} \cap (\L_{s-r+1} \cup \cdots \cup \L_{s-1})$.
\end{enumerate}
This configuration then gives the desired result.
\end{proof}

\begin{rem}
As mentioned in the introduction, $\k$-configurations of points
can be defined in $\mathbb{P}^n$ (see, e.g., \cite{GHS:1,GHS:2,GS,H}).
It is natural to ask if a result similar to Theorem 
\ref{maintheorem} also holds more generally.   Based upon some calculations,
it appears that this may be the case.  For example, let $\X$ be the 
$\k$-configuration of points in $\mathbb{P}^3$ found in 
\cite[Example 4.1]{GS} (see \cite{GS} for both the definition and
a picture).    For this example, one can see that there are three 
lines that contain four points.   The Hilbert function of $2\X$ is given by
\[
\begin{array}{lllllllllllllllllll}
\H_{2\X} & : &  1 & 4 & 10 & 20 & 35 & 50 & 57 & 60 & \ra.
\end{array}\]
Note that  ${\rm ri}(2\X) = 7$.  Also, we have 
$\Delta\H_{2\X}(7)=3$, i.e.,  
the same as the number of lines containing four points,  which is
similar to our statement in Theorem \ref{restatement}.

Although we suspect that a more general result holds, our proof
relies on techniques developed in \cite{CHT} that only give
precise information when the points are in $\mathbb{P}^2$.  
\end{rem}

\noindent
{\bf Acknowledgements} We would like to thank the referee for their helpful comments and suggestions.
Work on this project began when the second and third authors
visited Anthony (Tony) V. Geramita in Kingston, Ontario in the summer of 2015.
Computer
experiments using {\tt CoCoA} \cite{cocoa} inspired our main result. 
Unfortunately, Tony became quite ill soon after our visit, and he passed 
away in June 2016.  We would like to thank Tony for the input he was able 
to provide during the very initial stage of this project.
Shin's research was supported by the Basic Science Research Program of the NRF (Korea) under grant (No. 2016R1D1A1B03931683). Van Tuyl's research was supported 
in part by NSERC Discovery Grant 2014-03898.



\begin{thebibliography}{10}
\bibitem{cocoa} J. Abbott, A. Bigatti, G. Lagorio,
CoCoA-5: a system for doing Computations in Commutative Algebra.
Available at {\tt http://cocoa.dima.unige.it}

\bibitem{BGM}
A. Bigatti, A.V. Geramita, J. Migliore, 
Geometric consequences of extremal behavior in a theorem of Macaulay.
Trans. Amer. Math. Soc. {\bf 346} (1994), 203--235.
 
\bibitem{CTV}
M.V. Catalisano, N.V. Trung, G. Valla, 
A sharp bound for the regularity index of fat points in general position.
Proc. Amer. Math. Soc. {\bf 118} (1993), 717--724. 

\bibitem{CM}
L. Chiantini, J. Migliore,
Almost maximal growth of the Hilbert function.
J. Algebra {\bf 431} (2015), 38--77.

\bibitem{CHT}
S. Cooper, B. Harbourne, Z. Teitler, 
Combinatorial bounds on Hilbert functions of 
fat points in projective space. 
J. Pure Appl. Algebra {\bf 215} (2011), 2165--2179.

\bibitem{GHM}
A.V. Geramita, B. Harbourne, J. Migliore,
Star configurations in $\mathbb{P}^n$.
J. Algebra {\bf 376} (2013), 279--299.
 
\bibitem{GHS:1}
A.V. Geramita, T. Harima, Y.S. Shin, 
An alternative to the Hilbert function for the ideal of a 
finite set of points in $\mathbb{P}^n$.
Illinois J. Math. {\bf 45} (2001), 1--23. 


\bibitem{GHS:5} A.V. Geramita, T. Harima, Y.S. Shin, 
Decompositions of the Hilbert function of a set of points in $\P^n$, 
Canadian J. Math. {\bf 53} (2001), 923--943.

\bibitem{GHS:2}
A.V. Geramita, T. Harima, Y.S. Shin, 
Extremal point sets and Gorenstein ideals.
Adv. Math. {\bf 152} (2000), 78--119.

\bibitem{GMS}
A.V. Geramita, J. Migliore, L. Sabourin, 
On the first infinitesimal neighborhood of a 
linear configuration of points in $\mathbb{P}^2$.
J. Algebra {\bf 298} (2006), 563--611.
 
\bibitem{GS}
A.V.  Geramita, Y.S. Shin, $k$-configurations in 
$\mathbb{P}^3$ all have extremal resolutions. 
J. Algebra {\bf 213} (1999), 351--368. 

\bibitem{H}
T. Harima, 
Some examples of unimodal Gorenstein sequences. 
J. Pure Appl. Algebra {\bf 103} (1995), 313--324. 

\bibitem{RR}
L.G. Roberts, M. Roitman, 
On Hilbert functions of reduced and of integral algebras.
J. Pure Appl. Algebra {\bf 56} (1989), 85--104. 


\bibitem{ZS}
O.\ Zariski, P.\ Samuel,
{\it Commutative Algebra, vol. II}, Springer, 1960.
\end{thebibliography}
\end{document}